\documentclass[hidelinks,onefignum,onetabnum]{siamart251216}

\usepackage{amsmath,amssymb,mathtools}
\usepackage{enumitem}
\usepackage{microtype}
\usepackage{xcolor}
\usepackage{mathrsfs}
\hypersetup{colorlinks=true,linkcolor=blue!50!black,citecolor=blue!50!black,urlcolor=blue!50!black}
\newcommand{\eps}{\varepsilon}
\newcommand{\sbvto}{\xrightarrow{\mathrm{SBV}}}

\newcommand{\kap}{\kappa_\varepsilon}

\newcommand{\mres}{\mathbin{\vrule height 1.6ex depth 0pt width .13ex\vrule height .13ex depth 0pt width 1.3ex}}

\newtheorem{remark}{Remark}

\ifpdf
\hypersetup{
  pdftitle={An Example Article},
  pdfauthor={D. Doe, P. T. Frank, and J. E. Smith}
}
\fi

\begin{document}
\title{Simultaneous Sharp-Interface Limits for a Diffuse-Interface Landau--de Gennes Droplet Model\thanks{\funding{W.W was supported by the National Natural Science Foundation of China under Grant No.~12271476. L.Z. was supported by the National Natural Science Foundation of China (No. 12225102, T2321001, and 12288101) and the National Key Research and Development Program of China  2024YFA0919500.}}}
\author{Sisi Guan\thanks{Beijing International Center for Mathematical Research, Peking University, Beijing 100871, China (email: sisiguan98@bicmr.pku.edu.cn).}\and Wei Wang\thanks{School of Mathematical Sciences, Zhejiang University, Hangzhou 310058, China (email: wangw07@zju.edu.cn).}\and Lei Zhang\thanks{School of Mathematical Sciences, Beijing International Center for Mathematical Research, Center for Quantitative Biology, Center for Machine Learning Research, Institute for Artificial Intelligence, Peking University, Beijing 100871, China (zhangl@math.pku.edu.cn).}}

\headers{Simultaneous Sharp-Interface Limits}{S. Guan, W. Wang, and L. Zhang}
\maketitle

\begin{abstract}
Nematic droplet models based on diffuse interfaces involve competing microscopic transition layers whose interaction is lost if the small parameters are sent to zero successively. We study this interaction for the diffuse-interface Landau--de Gennes model of Shi et al., with phase-field thickness \(\varepsilon\) and exterior \(Q\)-tensor relaxation length \(\kappa_\varepsilon\), and show that their relative scale selects the effective sharp-interface theory. Within a regular-interface class, we establish the corresponding sharp-interface limit for every limiting ratio $r=\lim_{\varepsilon\to0}\kap/\eps\in[0,+\infty]$. Three qualitatively different regimes emerge. When \(r=+\infty\), the two transition layers separate asymptotically and the interfacial costs decouple, recovering the previously obtained iterated limit. When \(r\in(0,+\infty)\), the layers overlap and produce a genuinely coupled, \(r\)-dependent anisotropic surface density. When \(r=0\), finite energy forces the interior trace of \(Q\) to vanish on the limiting interface. The comparable-scale regime requires a genuinely coupled interfacial analysis. Thus the exterior relaxation length is not merely a regularization parameter: together with the phase-field thickness, it determines the macroscopic surface energy and anchoring law of the sharp-interface limit.
\end{abstract}



\section{Introduction}\label{sec:intro}
Liquid crystals are states of matter intermediate between liquids and crystalline solids \cite{chandrasekhar1992liquid,stewart2004continuum,virga1994variational}. Among the various liquid-crystal phases, the nematic phase is the simplest and most widely studied. Its constituent molecules have no long-range positional order, but their long axes tend to align along locally preferred directions~\cite{deGennesProst1993,frank1958theory,onsager1949effects,oseen1933theory,wang2021modelling}. A nematic droplet is a finite region of nematic liquid crystal surrounded by an isotropic phase, with an interface separating the two phases. Its molecular orientation interacts with the interface geometry, leading to distinctive physical behaviour. In particular, in polymer-dispersed liquid-crystal systems, their electro-optical properties have been exploited in display and switching technologies \cite{doane1988polymer,west1988phase}. More broadly, they provide a natural variational setting for studying the interaction between orientational order, topological defects, and free interfaces in anisotropic materials~\cite{kaznacheev2003tactoids,kim2013morphogenesis,LeoniEtAl2017,nastishin2005optical,prinsen2003tactoids,tortora2011chiral,wang2018tactoids}.

A fundamental modeling issue for nematic droplets is how to represent the interface separating the nematic and isotropic phases. In sharp-interface formulations, the droplet is described by an unknown region, and the energy typically consists of bulk elastic contributions in the droplet together with surface anchoring energies on its boundary, subject to a volume constraint~\cite{GengLin2022,LinPoon1996,Virga1989}. Since both the orientational order and the droplet region are part of the variational problem, such formulations lead to analytically challenging free-boundary problems. Diffuse-interface models instead introduce a phase field that varies smoothly between the nematic and isotropic phases, allowing the problem to be formulated on a fixed domain without explicitly representing the free boundary \cite{ForestWangYang2012,MataGarciaCerveraCeniceros2014,WuShiZhangImproved,WuEtAlDiffuseInterface}. This raises a natural variational question: \emph{does the diffuse-interface model converge to an effective sharp-interface description as the interfacial thickness tends to zero, and what surface energy and anchoring law emerge in the limit?}


Our starting point is the improved diffuse-interface Landau--de Gennes model of Shi et al.~\cite{WuShiZhangImproved}. In this model, a phase field distinguishes the nematic and isotropic regions and resolves their diffuse interface, while an additional relaxation mechanism drives the \(Q\)-tensor toward the isotropic state outside the droplet. These two mechanisms are associated with two distinct interfacial length scales: the phase field varies across a layer of thickness \(O(\varepsilon)\), whereas the exterior \(Q\)-tensor relaxes on a scale \(O(\kappa)\). Consequently, the sharp-interface limit depends not only on $\eps,\kappa\to 0$, but also on their relative scaling.

The analysis in \cite{WuShiZhangImproved} considered the iterated limit obtained by first letting \(\varepsilon\to0\) with \(\kappa>0\) fixed and subsequently sending \(\kappa\to0\). In terms of the relative scaling, this corresponds to the regime $\kappa/\eps\to \infty$, where the two transition scales become asymptotically separated. It leaves open a more general question: \emph{what sharp-interface energy emerges when \(\varepsilon\) and \(\kappa\) vanish simultaneously, and how does this energy change when the two interfacial layers separate, overlap, or occur on comparable scales?} We address this question through a simultaneous \(\Gamma\)-convergence analysis, which allows the ratio $r_\eps=\kap/\eps$ to converge to any value in \([0,+\infty]\).

Our main result shows that these three scaling regimes lead to qualitatively different sharp-interface theories. More precisely, within the regular-interface class, we establish the \(\Gamma\)-convergence of \(E_{\varepsilon,\kappa_\varepsilon}\) to a family of sharp-interface energies covering \emph{all relative scalings of the two diffuse length scales}. When \(r_\varepsilon\to+\infty\), the two interfacial mechanisms are asymptotically separated, and their contributions become additive, recovering the surface energy obtained from the iterated limiting procedure in \cite{WuShiZhangImproved}. When \(r_\varepsilon\to\ell\in(0,+\infty)\), the two transitions occur on comparable scales and interact within the same interfacial layer, producing a \emph{genuinely coupled, \(\ell\)-dependent surface density}. When \(r_\varepsilon\to0\), the \(Q\)-tensor relaxes on a much shorter scale, and finite energy \emph{forces the interior trace of \(Q\) to vanish on the limiting interface}. Thus, the exterior relaxation length is not merely a regularization parameter: together with the phase-field thickness, it selects the effective surface energy and anchoring law.

The main analytical difficulties arise from three features of the model. First, the elastic energy controls only the weighted quantity \(\phi_\varepsilon\nabla Q_\varepsilon\), so small low-phase sets may carry large gradients and prevent uniform local \(H^1\) control of \(Q_\varepsilon\). To overcome this degeneracy, we combine the high-phase restriction strategy of \cite{lin2012phase} with radial truncation and a diagonal passage in the truncation level. This yields uniformly bounded tensor-valued \(SBV^2\) approximations and, crucially, \emph{strong \(L^1\) compactness of the original, untruncated sequence}, without assuming a uniform \(L^\infty\) bound or local weak \(H^1\) compactness. \emph{The novelty lies in combining superlevel-set restriction with tensor-valued truncation and using the resulting compactness in the transition-measure analysis}.

The second obstacle appears in the critical regime \(r_\varepsilon\to\ell\in(0,+\infty)\), where the phase transition and exterior tensor relaxation contribute to the same transition layer. The appropriate Modica-type primitive depends on \(Q_\varepsilon\), and differentiating it creates an additional \(\nabla Q_\varepsilon\)-term precisely where direct \(H^1_{\mathrm{loc}}\) compactness is unavailable. The truncated-\(SBV^2\) framework supplies the distributional compactness needed to identify \(Q_\varepsilon D\Phi(\phi_\varepsilon)\), control this additional term, and derive the anisotropic, \(\ell\)-dependent interfacial density.

Finally, small high-phase incursions into the exterior locally destroy the coercivity of the void potential. We prove a \emph{one-dimensional capacity comparison lemma} showing that, among measurable exceptional sets of prescribed measure, the least favorable configuration is an interval adjacent to the interface. This geometry-independent estimate recovers the sharp exterior relaxation cost without requiring pointwise convergence of traces and may be useful for other singularly perturbed energies with localized loss of coercivity.

The remainder of the paper is organized as follows. Section \ref{sec:model} introduces the diffuse-interface energy, the three relative-scaling regimes, and the corresponding sharp-interface functionals, and states the main result. Section \ref{sec:preliminaries} develops the compactness framework, including the truncated-\(SBV^2\) construction, and records the geometric preliminaries near the limiting interface. Section~\ref{sec:lower-bound} proves the lower bounds through transition-measure arguments and the one-dimensional capacity estimate. Section~\ref{sec:upper-bound} constructs boundary-dependent recovery sequences and proves the upper bounds. Section~\ref{sec:conclusion} summarizes the results and discusses possible extensions.

\section{Model and main results}\label{sec:model}
\subsection{Diffuse-interface energy}
Throughout the paper, $\Omega\subset\mathbb{R}^{3}$ is a bounded Lipschitz domain, $V_{0}\in(0,|\Omega|)$ is prescribed, and $\lambda,w_p,w_v,\omega>0$. We let $\varepsilon\downarrow0$ and prescribe $\kappa_{\varepsilon}\downarrow0$ as a function of $\varepsilon$. We denote by
\[
    \mathcal{S}_{0}:=\left\{Q\in\mathbb{R}^{3\times3}:Q^{T}=Q,\ \operatorname{tr}Q=0\right\}
\]
the space of symmetric traceless tensors.

For $\varepsilon>0$, we consider $(Q_{\varepsilon},\phi_{\varepsilon})\in H^{1}(\Omega;\mathcal{S}_{0})\times H_{0}^{1}(\Omega)$. Here $Q_{\varepsilon}$ describes the orientational order, while $\phi_{\varepsilon}$ is a phase field with $\phi_{\varepsilon}\approx1$ in the nematic droplet and $\phi_{\varepsilon}\approx0$ in the exterior isotropic phase. The boundary condition $\phi_{\varepsilon}\in H_{0}^{1}(\Omega)$ imposes the exterior phase on $\partial\Omega$, consistently with the assumption that the limiting droplet is compactly contained in $\Omega$.

The diffuse-interface energy is
\begin{equation}\label{eq:diffuse-energy}
    E_{\varepsilon,\kappa_{\varepsilon}}(Q_{\varepsilon},\phi_{\varepsilon}):=E_{\varepsilon}^{\mathrm{bulk}}+E_{\varepsilon}^{\mathrm{mix}}+E_{\varepsilon}^{\mathrm{void}},
\end{equation}
where the three contributions are described below. We consider the improved diffuse-interface Landau--de Gennes droplet model introduced by Shi et al. \cite{WuShiZhangImproved}; we recall only the ingredients needed for the asymptotic analysis below.

\paragraph{The bulk contribution}
We set
\[
    E_{\varepsilon}^{\mathrm{bulk}}:=\int_{\Omega}\phi_{\varepsilon}^{2}\left(|\nabla Q_{\varepsilon}|^{2}+\frac{1}{2\lambda}F_{b}(Q_{\varepsilon})\right)\,dx,
\]
where
\[
    F_{b}(Q):=\frac{a}{2}\operatorname{tr}(Q^{2})-\frac{b}{3}\operatorname{tr}(Q^{3})+\frac14\bigl(\operatorname{tr}(Q^{2})\bigr)^{2}-B_{\min}.
\]
We assume that $b>0$ and $a<b^{2}/27$, and choose $B_{\min}$ so that $F_{b}\geq0$ and
\[
    F_{b}^{-1}(0)=\mathcal{N},
\]
where
\[
    \mathcal{N}:=\left\{s_{+}\left(n\otimes n-\frac13 I\right):n\in\mathbb{S}^{2}\right\},\qquad s_{+}=\frac{b+\sqrt{b^{2}-24a}}{4}.
\]
Thus, $\mathcal{N}$ is the uniaxial nematic well of the Landau--de Gennes bulk potential.

The factor \(\phi_\varepsilon^2\) localizes the Landau--de Gennes energy to the nematic region and, importantly for the analysis below, yields only weighted control of \(\nabla Q_\varepsilon\).

\paragraph{The mixing contribution}
We set
\[
    E_{\varepsilon}^{\mathrm{mix}}:=w_{p}\int_{\Omega}\left(\varepsilon^{-1}W(\phi_{\varepsilon})+\varepsilon|\nabla\phi_{\varepsilon}|^{2}+\omega\varepsilon\left|\left(Q_{\varepsilon}+\frac{s_{+}}{3}I\right)\nabla\phi_{\varepsilon}\right|^{2}\right)\,dx,
\]
where \(W(t):=t^{2}(1-t)^{2}\) is the standard double-well potential. The first two terms constitute the Modica--Mortola energy and generate a diffuse interface of thickness \(O(\varepsilon)\), while the last term is the anisotropic contribution introduced in \cite{WuShiZhangImproved} to encode planar anchoring. For \(Q=s_{+}(n\otimes n-\frac13 I)\in\mathcal N\),
\[
    \left|\left(Q+\frac{s_{+}}{3}I\right)\nabla\phi\right|^{2}=s_{+}^{2}(n\cdot\nabla\phi)^{2},
\]
and hence this contribution favors \(n\cdot\nu=0\) in the sharp-interface limit. Its dependence on \(Q_\varepsilon\) within the \(O(\varepsilon)\) transition layer is responsible for the anisotropic \(Q\)-dependent surface density appearing below.

\paragraph{The exterior-relaxation contribution}
Finally, we set
\[
    E_{\varepsilon}^{\mathrm{void}}:=\frac{w_{v}}{2}\int_{\Omega}\left(\kappa_{\varepsilon}|\nabla Q_{\varepsilon}|^{2}+\kappa_{\varepsilon}^{-1}(1-\phi_{\varepsilon})^{2}|Q_{\varepsilon}|^{2}\right)\,dx.
\]
This term drives \(Q_{\varepsilon}\) toward the isotropic state \(Q=0\) in the exterior phase and produces a relaxation layer of thickness \(O(\kappa_{\varepsilon})\). The interaction between this layer and the \(O(\varepsilon)\) phase-field transition layer is a central feature of the present analysis. In particular, the limiting interfacial behavior depends on the relative scale
\[
r_{\varepsilon}:=\frac{\kappa_{\varepsilon}}{\varepsilon},
\]
leading to the distinct sharp-interface regimes introduced in the next subsection.

We impose the admissibility constraint
\[
    \mathcal{A}_{\varepsilon}:=\left\{(Q_{\varepsilon},\phi_{\varepsilon})\in H^{1}(\Omega;\mathcal{S}_{0})\times H_{0}^{1}(\Omega):\ 0\leq\phi_{\varepsilon}\leq1\ \text{a.e. in }\Omega,\quad\int_{\Omega}\phi_{\varepsilon}\,dx=V_{0}\right\}.
\]
The pointwise constraint reflects the interpretation of $\phi_{\varepsilon}$ as a phase fraction, while the integral constraint fixes the droplet volume.

\subsection{Relative scaling and effective interfacial densities}
\label{subsec:relative-scaling}

The simultaneous limit is governed by the relative scale
\[
    r_\eps:=\frac{\kappa_\eps}{\eps}.
\]
After passing to a subsequence if necessary, we assume throughout that
\[
    r_\eps\to r\in[0,+\infty].
\]
The regimes \(r=0\), \(r\in(0,+\infty)\), and \(r=+\infty\) correspond, respectively, to \(\kappa_\eps\ll\eps\), \(\kappa_\eps\sim\eps\), and \(\eps\ll\kappa_\eps\).

For later use, set
\begin{equation}\label{eq:anchoring-matrix}
    A(q):=q+\frac{s_+}{3}I,
\end{equation}
and define, for \(q\in\mathcal S_0\) and \(\xi\in\mathbb R^3\),
\[
    F(q,\xi):=\left(|\xi|^2+\omega|A(q)\xi|^2\right)^{1/2},\qquad G(q,\nu):=F(q,\nu),\quad \nu\in\mathbb S^2.
\]
In particular,
\begin{equation}\label{eq:anisotropic-norm-growth}
    F(q,\xi)\le C(1+|q|)|\xi|.
\end{equation}

We now define the transition-layer contribution \(m_r\). For
\(\ell\in(0,+\infty)\), let
\[
    m_\ell(q,\nu):=2\sqrt{w_p}\,G(q,\nu)\int_0^1\left(w_pW(s)+\frac{w_v}{2\ell}|q|^2(1-s)^2\right)^{1/2}\,ds.
\]
At the two endpoint regimes, set
\[
    m_\infty(q,\nu):=\frac{w_p}{3}G(q,\nu),
\]
and
\[
    m_0(q,\nu):=
    \begin{cases}
    \displaystyle
    \frac{w_p}{3}
    \left(1+\frac{\omega s_+^2}{9}\right)^{1/2},
        & q=0,\\[2mm]
    +\infty, & q\neq0.
    \end{cases}
\]
The effective interfacial density is then defined uniformly by
\begin{equation}\label{eq:gamma-r}
    \gamma_r(q,\nu):=m_r(q,\nu)+\frac{w_v}{2}|q|^2.
\end{equation}

These three cases reflect the interaction of the two diffuse layers. For \(r=+\infty\), the phase-field transition and the exterior \(Q\)-relaxation are asymptotically separated, and their costs add. For \(r=\ell\in(0,+\infty)\), the two mechanisms interact within the same transition layer, producing an \(\ell\)-dependent coupled surface density. For \(r=0\), finite interfacial energy requires \(q=0\); consequently, finite limiting energy enforces
\[
    Q^-=0\qquad \mathcal H^2\text{-a.e. on }\partial A.
\]

For each fixed \(r\in(0,+\infty]\), the map \((q,\nu)\mapsto\gamma_r(q,\nu)\) is continuous and satisfies
\[
    0\le \gamma_r(q,\nu)\le C_r(1+|q|^2).
\]
Moreover, for fixed \((q,\nu)\), \(\gamma_\ell(q,\nu)\) is nonincreasing in \(\ell\), with $\gamma_\ell(q,\nu)\to\gamma_\infty(q,\nu)$ as $\ell\to+\infty$, while, as \(\ell\downarrow0\), $\gamma_\ell(0,\nu)\to\gamma_0(0,\nu)$ and $\gamma_\ell(q,\nu)\to+\infty$ for $q\neq 0$. Thus the critical densities interpolate consistently between the two endpoint regimes.

\subsection{Regular configurations and main results}
We formulate the sharp-interface result on the following regular class:
\[
    \mathcal A_0^{\rm reg}:=\left\{(Q,A):
    \begin{array}{l}
        A\Subset\Omega\ \text{is open with }C^2\text{ boundary},\\
        |A|=V_0,\qquad Q\in H^1(A;\mathcal{S}_0)
    \end{array}
    \right\}.
\]
Whenever $(Q,A)\in\mathcal A_0^{\rm reg}$, we identify $Q$ with its zero extension to $\Omega$. In particular, $Q\in L^2(\Omega;\mathcal{S}_0)$ and $Q=0$ a.e. in $A^c$. The trace theorem yields a well-defined interior trace
\[
    Q^-\in H^{1/2}(\partial A;\mathcal S_0) \subset L^2(\partial A;\mathcal S_0).
\]
We denote by $\nu_A$ the exterior unit normal to $\partial A$.

For $(Q,A)\in\mathcal A_0^{\rm reg}$ and $r\in[0,+\infty]$, define
\begin{equation}\label{eq:E-limit}
    E_0^r(Q,A):=\int_A\left(|\nabla Q|^2+\frac{1}{2\lambda}F_b(Q)\right)\,dx+\int_{\partial A}\gamma_r(Q^-,\nu_A)\,d\mathcal H^2.
\end{equation}
This is the regime-dependent sharp-interface functional identified in the simultaneous limit.

We write $(Q_\eps,\phi_\eps)\to(Q,A)$ if
\begin{align*}
    \phi_\eps\to\chi_A\quad\text{strongly in }L^1(\Omega),\qquad Q_\eps\rightharpoonup Q
    \quad\text{weakly in }L^2(\Omega;\mathcal{S}_0),
\end{align*}
and
\[
    Q_\eps\to0 \quad\text{strongly in }L^2(A^c;\mathcal{S}_0).
\]

\begin{theorem}\label{thm:main}
Let $\eps,\kap\to0$ with $r_\eps=\kappa_\eps/\eps\to r\in [0,+\infty]$. Then, within the regular class $\mathcal A_0^{\rm reg}$, the sharp-interface functional $E_0^r$ defined in \eqref{eq:E-limit} provides matching lower and upper bounds in the sense stated below.
\begin{enumerate}
\item[\rm (i)]
Let $(Q_\eps,\phi_\eps)\in\mathcal A_\eps$ be an energy-bounded sequence such that $(Q_\eps,\phi_\eps)\to (Q,A)$ for some $(Q,A)\in\mathcal A_0^{\rm reg}$. Then
\[
    \liminf_{\eps\to0}E_{\eps,\kappa_\eps}(Q_\eps,\phi_\eps)\geq E_0^r(Q,A).
\]

\item[\rm (ii)]
For every $(Q,A)\in\mathcal A_0^{\rm reg}$ satisfying $E_0^r(Q,A)<+\infty$, there exists $(Q_\eps,\phi_\eps)\in\mathcal A_\eps$ such that
\[
    \phi_\eps\to\chi_A\quad\text{strongly in }L^1(\Omega),\qquad Q_\eps\to Q \quad\text{strongly in }L^2(\Omega;\mathcal{S}_0),
\]
and
\[
    \limsup_{\eps\to0}E_{\eps,\kappa_\eps}(Q_\eps,\phi_\eps)\leq E_0^r(Q,A).
\]
\end{enumerate}

Hence the effective sharp-interface energy is selected by the limiting ratio $r=\lim_{\eps\to0}\kappa_\eps/\eps$.
\end{theorem}

\section{Compactness and Geometric Preliminaries}
\label{sec:preliminaries}
For $s\in [0,1]$, define $\Phi(s):=\int_0^s \sqrt{W(t)}\mathrm{d}t$, then $c_0=\Phi(1)=1/6$.

\subsection{Compactness of bounded-energy sequences}
\label{sec:compactness}
\begin{proposition}[Phase-field compactness]
\label{prop:compactness}
Let $\eps,\kap\to0$ and let $(Q_\eps,\phi_\eps)\in\mathcal{A}_\eps$ be an energy-bounded sequence. Then, after passing to a subsequence, there exist a set $A\subset\Omega$ of finite perimeter, with $|A|=V_0$, and a tensor field $Q\in L^2(\Omega;\mathcal{S}_0)$ with $Q=0$ a.e. in $A^c$, such that
\begin{align}
    &\phi_\varepsilon\to\chi_A\quad \text{strongly in }L^1(\Omega),\label{eq:phase-field-L1}\\
    &Q_\varepsilon\rightharpoonup Q\quad \text{weakly in }L^2(\Omega; \mathcal{S}_0),\qquad  Q_\varepsilon\to0\quad\text{strongly in }L^2(A^c;\mathcal S_0).\label{eq:tensor-L2}
\end{align}
Moreover,
\begin{align}\label{eq:Phi-phase-L1}
    \Phi(\phi_\eps)\to c_0\chi_A \qquad\text{strongly in }L^1(\Omega),
\end{align}
and
\begin{equation}\label{eq:phase-transition-measure}
    D\Phi(\phi_\eps) \stackrel{*}{\rightharpoonup} c_0D\chi_A \qquad \text{in }\mathcal M(\Omega;\mathbb R^3).
\end{equation}
\end{proposition}
\begin{proof}
The phase-field energy yields the uniform BV bound $$\sup_{\eps>0}|D\Phi(\phi_\eps)|(\Omega)<+\infty,$$ and hence the standard Modica–Mortola compactness argument \cite{modica1977gamma,modica1987gradient,sternberg1988effect} gives $\phi_\eps\to\chi_A$ strongly in $L^1(\Omega)$ for a set $A$ of finite perimeter. The volume constraint gives $|A|=V_0$. Since $\Phi$ is continuous on $[0,1]$, $\Phi(\phi_\eps)\to c_0\chi_A$ strongly in $L^1(\Omega)$. The uniform $BV$ bound then implies $D\Phi(\phi_\eps) \stackrel{*}{\rightharpoonup} c_0D\chi_A.$

It remains to obtain compactness for $Q_\eps$. The bulk potential is coercive: there exist $c_b,C_b>0$ such that
\begin{equation}\label{eq:bulk-potential-coercivity}
    F_b(q)\geq c_b|q|^4-C_b \qquad\text{for every }q\in \mathcal{S}_0.
\end{equation}
Since $0\leq\phi_\varepsilon\leq1$, the energy bound and \eqref{eq:bulk-potential-coercivity} imply
\begin{equation}\label{eq:weighted-tensor-L4}
    \sup_{\varepsilon>0}\int_\Omega\phi_\varepsilon^2|Q_\varepsilon|^4\,dx<+\infty.
\end{equation}

Set $H_\eps:=\{\phi_\eps\geq\tfrac12\}$ and $L_\eps:=\{\phi_\eps<\tfrac12\}$. On $H_\eps$, the bulk-energy bound and $\phi_\eps^2\geq1/4$ imply
\[
    \int_{H_\eps}|Q_\eps|^4\,dx\leq C.
\]
On $L_\eps$, the void-potential bound gives
\[
    \int_{L_\eps}|Q_\eps|^2\,dx \leq4\int_\Omega(1-\phi_\eps)^2|Q_\eps|^2\,dx \leq C\kappa_\eps.
\]
Thus $Q_\eps$ is bounded in $L^2(\Omega)$ and, along a subsequence, $Q_\eps\rightharpoonup Q$ weakly in $L^2(\Omega;\mathcal{S}_0)$ for some $Q\in L^2(\Omega;\mathcal{S}_0)$.

We finally prove the strong convergence in the exterior phase. Fix $\eta\in(0,1)$. We decompose
\[
    \int_{A^c}|Q_\eps|^2\,dx =\int_{A^c\cap\{\phi_\eps<\eta\}}|Q_\eps|^2\,dx+\int_{A^c\cap\{\phi_\eps\geq\eta\}}|Q_\eps|^2\,dx.
\]
The first term is bounded by $C\kappa_\eps/(1-\eta)^2$ and hence vanishes. Coercivity of $F_b$ and $\phi_\eps\geq\eta$ give
\[
    \int_{\{\phi_\eps\geq\eta\}}|Q_\eps|^4\,dx \leq C_\eta.
\]
Since
\[
    |A^c\cap\{\phi_\eps\geq\eta\}|\leq\frac1\eta\int_{A^c}\phi_\eps\,dx \to0,
\]
H\"older's inequality yields
\[
    \int_{A^c\cap\{\phi_\eps\geq\eta\}}|Q_\eps|^2\,dx \to0.
\]
Therefore,
\[
    Q_\eps\to0 \qquad\text{strongly in }L^2(A^c;\mathcal{S}_0).
\]
In particular, the weak limit satisfies $Q=0$ a.e. in $A^c$.
\end{proof}

For later use, we introduce the notation associated with the truncation-and-restriction argument. For $M>0$, define the radial truncation $T_M:\mathcal S_0\to\mathcal S_0$ by
\[
    T_M(q):=
    \begin{cases}
    q, & |q|\leq M,\\[1mm]
    M\dfrac{q}{|q|}, & |q|>M.
    \end{cases}
\]
The map $T_M$ is $1$-Lipschitz, satisfies $T_M(0)=0$, and obeys $|T_M(q)|\leq M$ for every $q\in\mathcal S_0$. Moreover, by the Sobolev chain rule, for every $Q\in H^1(\Omega;\mathcal S_0)$,
\begin{equation}
    |\nabla T_M(Q)|\leq |\nabla Q|\quad\text{a.e. in }\Omega,\label{eq:truncation-gradient}
\end{equation}
and
\begin{equation}\label{eq:truncation-gradient-error}
    |\nabla Q-\nabla T_M(Q)|\leq2|\nabla Q|\,\mathbf 1_{\{|Q|\geq M\}} \qquad\text{a.e. in }\Omega.
\end{equation}

For $t\in(0,1)$, set
\[
    \Omega^+_{\varepsilon,t}:=\{\phi_\varepsilon>t\}, \qquad V_{\varepsilon,t}^M:=T_M(Q_\varepsilon)\chi_{\Omega_{\varepsilon,t}^+}.
\]
Once $t_\eps$ is selected, abbreviate these as $\Omega_\eps^+$ and $V_\eps^M$.

For $V\in SBV(\Omega;\mathcal S_0)$, we denote by $J_V$ its approximate jump set, by $V^\pm$ its one-sided approximate traces, and by $\nu_V$ the corresponding measure-theoretic normal. Its distributional derivative has the decomposition
\[
    DV=\nabla V\,\mathcal L^3+(V^+-V^-)\otimes\nu_V\,\mathcal H^2\mres J_V,
\]
where $\nabla V$ denotes the density of the absolutely continuous part of $DV$. We write
\[
    SBV^2(\Omega;\mathcal S_0):=\left\{V\in SBV(\Omega;\mathcal S_0):\nabla V\in L^2(\Omega;\mathcal S_0\otimes\mathbb R^3),\quad\mathcal H^2(J_V)<+\infty\right\}.
\]

We now implement the truncated-\(SBV^2\) compactness framework announced in the introduction. The selection of high-phase superlevel sets follows the restriction strategy of Lin, Pan, and Wang \cite{lin2012phase}, while the radial truncation is used to obtain uniformly bounded tensor-valued \(SBV^2\) approximations. A key point is that the resulting estimates allow us to pass from compactness of each fixed truncation to strong \(L^1\) compactness of the original sequence \(Q_\varepsilon\).

\begin{proposition}[Truncated-tensor $SBV^2$ compactness] \label{prop:truncated-sbv-compactness}
Let $\varepsilon,\kappa_\eps\to0$, and let $(Q_\varepsilon,\phi_\varepsilon)\in\mathcal A_\varepsilon$ be an energy-bounded sequence satisfying $(Q_\varepsilon,\phi_\varepsilon)\to(Q,A)$. Fix $0<a<b<1$. Then there exist levels $t_\varepsilon\in(a,b)$ such that the associated sets $\Omega_\varepsilon^+$ satisfy
\begin{equation}\label{eq:high-phase-perimeter}
    \sup_{\varepsilon>0}P(\Omega_\varepsilon^+;\Omega)<+\infty,\qquad \chi_{\Omega^+_\varepsilon}\to\chi_A \quad\text{strongly in }L^1(\Omega).
\end{equation}

For every $M>0$, $V_\varepsilon^M\in SBV^2(\Omega;\mathcal S_0)\cap L^\infty(\Omega;\mathcal S_0)$ and
\begin{equation}
    \int_\Omega |\nabla V_\varepsilon^M|^2\,dx \leq\frac{1}{a^2}\int_\Omega\phi_\varepsilon^2|\nabla Q_\varepsilon|^2\,dx,\label{eq:truncated-sbv-bounds}
\end{equation}
together with
\begin{equation}
    \mathcal H^2(J_{V_\varepsilon^M}) \leq P(\Omega^+_\varepsilon;\Omega),\qquad\|V_\varepsilon^M\|_{L^\infty(\Omega)}\leq M.\label{eq:truncated-jump-bound}
\end{equation}

In addition—and this is the compactness conclusion for the original sequence—after passing to a subsequence independent of \(M\),
\begin{equation}
    Q_\varepsilon\to Q\quad\text{strongly in }L^1(\Omega;\mathcal S_0),\label{eq:strong-l1-Q}
\end{equation}
and, for every $M\in\mathbb N$,
\begin{align}
    V_\varepsilon^M &\to T_M(Q) &&\text{strongly in }L^1(\Omega;\mathcal S_0),\label{eq:truncated-sbv-convergence}\\
    \nabla V_\varepsilon^M&\rightharpoonup \nabla T_M(Q) &&\text{weakly in } L^2(\Omega;\mathcal S_0\otimes\mathbb R^3).\label{eq:truncated-gradient-convergence}
\end{align}
Here $Q$ and $T_M(Q)$ are understood as their zero extensions to $\Omega$. In particular, $T_M(Q)\in SBV^2(\Omega;\mathcal S_0)$ and
\begin{align}
    \int_\Omega |\nabla T_M(Q)|^2\,dx &\leq \liminf_{\varepsilon\to0} \int_\Omega|\nabla V_\varepsilon^M|^2\,dx,\label{eq:truncated-gradient-lsc}\\
    \mathcal H^2(J_{T_M(Q)}) &\leq \liminf_{\varepsilon\to0}\mathcal H^2(J_{V_\varepsilon^M}).\label{eq:truncated-jump-lsc}
\end{align}

If, in addition, $A$ is open and
$Q\in H^1(A;\mathcal S_0)$, then $\nabla T_M(Q)=\chi_A\nabla T_M(Q|_A)$ a.e. in $\Omega$ and hence
\begin{equation}
    a^2\int_A|\nabla T_M(Q)|^2\,dx\leq\liminf_{\varepsilon\to0}\int_\Omega\phi_\varepsilon^2|\nabla Q_\varepsilon|^2\,dx\qquad\text{for every }M\in\mathbb N.\label{eq:truncated-weighted-gradient-liminf}
\end{equation}
\end{proposition}

\begin{proof}
Set $m_{a,b}:=\min_{s\in[a,b]}\sqrt{W(s)}>0$. By the coarea formula,
\begin{align*}
    \int_a^bP(\{\phi_\eps>t\};\Omega)\,dt=\int_{\{a<\phi_\eps<b\}}|\nabla\phi_\eps|\,dx\le\frac1{m_{a,b}}\int_\Omega\sqrt{W(\phi_\eps)}|\nabla\phi_\eps|\,dx\le C_{a,b}.
\end{align*}
Consequently, one may choose $t_\eps\in(a,b)$ such that $\Omega_\eps^+=\{\phi_\eps>t_\eps\}$ has finite perimeter and $P(\Omega^+_\eps;\Omega) \le C_{a,b}/(b-a)$, proving the first assertion in \eqref{eq:high-phase-perimeter}.

We next prove the convergence of the high-phase sets. Since $a\leq t_\eps\le b$, on $A\setminus \Omega^+_\eps$ one has $\phi_\eps\le b$ while $\phi_\eps\geq a$ on $\Omega_\eps^+\setminus A$, and hence
\[
    |A\setminus \Omega^+_\eps|\le\frac1{1-b}\int_A(1-\phi_\eps)\,dx,\qquad |\Omega^+_\eps\setminus A|\le\frac1a\int_{A^c}\phi_\eps\,dx.
\]
The strong $L^1$ convergence $\phi_\eps\to\chi_A$ therefore implies $|\Omega^+_\eps\triangle A|\to0$, which proves the second assertion in \eqref{eq:high-phase-perimeter}.

Since $T_M(Q_\eps)\in H^1(\Omega;\mathcal S_0)\cap L^\infty(\Omega;\mathcal S_0)$ and $\Omega^+_\eps$ has finite perimeter, the standard $BV$ product rule yields $V_\eps^M=T_M(Q_\eps)\chi_{\Omega^+_\eps}\in SBV(\Omega;\mathcal S_0)$ with $\nabla V_\eps^M=\chi_{\Omega^+_\eps}\nabla T_M(Q_\eps)$ a.e. in $\Omega$, and $J_{V_\eps^M}\subset\partial^*\Omega^+_\eps$ up to an $\mathcal H^2$-null set. Because $\phi_\eps\ge a$ on $\Omega^+_\eps$,
\begin{align*}
    \int_\Omega|\nabla V_\eps^M|^2dx&=\int_{\Omega^+_\eps}|\nabla T_M(Q_\eps)|^2dx\le\int_{\Omega^+_\eps}|\nabla Q_\eps|^2dx\le\frac1{a^2}\int_\Omega\phi_\eps^2|\nabla Q_\eps|^2dx.
\end{align*}
This proves \eqref{eq:truncated-sbv-bounds}. The inclusion of the jump sets and the pointwise bound $|T_M(Q_\eps)|\le M$ give \eqref{eq:truncated-jump-bound}.

For every fixed $M\in\mathbb N$, estimates \eqref{eq:truncated-sbv-bounds}--\eqref{eq:truncated-jump-bound} and the $SBV$ compactness theorem imply that $\{V_\eps^M\}_\eps$ is precompact in $L^1(\Omega;\mathcal S_0)$. By a diagonal argument, we may therefore assume that for every $M\in\mathbb N$ there exists $V^M\in SBV^2(\Omega;\mathcal S_0)$ such that
\begin{equation}
\label{eq:preliminary-V-limit}
    V_\eps^M\to V^M\quad\text{strongly in }L^1(\Omega;\mathcal S_0),
\end{equation}
and $\nabla V_\eps^M\rightharpoonup\nabla V^M$ weakly in $L^2(\Omega)$.

We now prove the strong $L^1$ compactness of the original tensor sequence. Since $T_M$ is the radial projection onto the closed ball of radius $M$, 
\[
    |q-T_M(q)|\le |q|\mathbf 1_{\{|q|>M\}}\le \frac1M|q|^2.
\]
It follows that
\begin{align}
    \|Q_\eps-V_\eps^M\|_{L^1(\Omega)}&\le\int_{\Omega\setminus \Omega^+_\eps}|Q_\eps|\,dx+\int_{\Omega^+_\eps}|Q_\eps-T_M(Q_\eps)|\,dx\nonumber\\
    &\le\|Q_\eps\|_{L^1(A^c)}+\|Q_\eps\|_{L^2(\Omega)}|A\setminus \Omega^+_\eps|^{1/2}+\frac1M\|Q_\eps\|_{L^2(\Omega)}^2.\label{eq:Q-V-error}
\end{align}
By Proposition~\ref{prop:compactness}, and $|A\setminus \Omega^+_\eps|\to0$, estimate \eqref{eq:Q-V-error} gives
\begin{equation}
\label{eq:Q-V-uniform-error}
    \limsup_{\eps\to0}\|Q_\eps-V_\eps^M\|_{L^1(\Omega)}\le\frac{C}{M}\qquad\text{for every }M\in\mathbb N.
\end{equation}

We now pass from compactness of the truncated restrictions to compactness of the original tensor sequence; this is the point at which the uniform truncation error is essential. We claim that $\{Q_\eps\}_\eps$ is a Cauchy sequence in $L^1(\Omega;\mathcal S_0)$ along the selected subsequence. Indeed, for every $M\in\mathbb N$,
\[
    \|Q_\eps-Q_{\eps'}\|_{L^1}\le\|Q_\eps-V_\eps^M\|_{L^1}+\|V_\eps^M-V_{\eps'}^M\|_{L^1}+\|V_{\eps'}^M-Q_{\eps'}\|_{L^1}.
\]
First let $\eps,\eps'\to0$ and use \eqref{eq:preliminary-V-limit}; then let $M\to\infty$ and use \eqref{eq:Q-V-uniform-error}. Thus there exists $\widetilde Q\in L^1(\Omega;\mathcal S_0)$ such that $Q_\eps\to\widetilde Q$ strongly in $L^1(\Omega;\mathcal S_0)$. On the other hand, $Q_\eps\rightharpoonup Q$ weakly in $L^2$, and hence also in the sense of distributions. The $L^1$ limit is therefore unique, and $\widetilde Q=Q$ a.e. in $\Omega$. This proves \eqref{eq:strong-l1-Q}.

We can now identify the limits in \eqref{eq:preliminary-V-limit}. Since $T_M$ is $1$-Lipschitz and $|T_M(q)|\le M$,
\begin{align*}
    \|V_\eps^M-T_M(Q)\chi_A\|_{L^1}&\le\|T_M(Q_\eps)-T_M(Q)\|_{L^1}+M\|\chi_{\Omega^+_\eps}-\chi_A\|_{L^1}\\
    &\le\|Q_\eps-Q\|_{L^1}+M\|\chi_{\Omega^+_\eps}-\chi_A\|_{L^1}\to0.
\end{align*}
Thus $V^M=T_M(Q)\chi_A$, which proves \eqref{eq:truncated-sbv-convergence} and \eqref{eq:truncated-gradient-convergence}. The lower-semicontinuity statements \eqref{eq:truncated-gradient-lsc} and \eqref{eq:truncated-jump-lsc} follow from the $SBV$ compactness theorem.

Finally, suppose that $A$ is open and $Q\in H^1(A;\mathcal S_0)$. Since $T_M$ is Lipschitz, $T_M(Q)\in H^1(A;\mathcal S_0)$, and the absolutely continuous part of the derivative of its zero extension is
\[
    \nabla\bigl(T_M(Q)\chi_A\bigr)=\chi_A\nabla T_M(Q)\qquad\text{a.e. in }\Omega.
\]
Combining \eqref{eq:truncated-gradient-lsc} with \eqref{eq:truncated-sbv-bounds} gives \eqref{eq:truncated-weighted-gradient-liminf}.
\end{proof}
\begin{remark}[Role of the truncated-\(SBV^2\) construction]
    The selection of high-phase superlevel sets is adapted from the restriction strategy in \cite{lin2012phase}. In the present problem, the elastic bound controls only \(\phi_\varepsilon\nabla Q_\varepsilon\), and therefore gives no direct unweighted \(H^1_{\mathrm{loc}}\) estimate for \(Q_\varepsilon\). The radial truncation serves two distinct purposes. First, it produces uniformly bounded \(SBV^2\) approximations to which the standard compactness theorem applies. Second, the truncation error is controlled uniformly in \(\varepsilon\) by the weighted bulk-potential bound, which permits the limit \(M\to\infty\) and yields strong \(L^1\) compactness of \(Q_\varepsilon\) itself. This distinction is important in Sections \ref{subsec:anisotropic} and \ref{subsec:coupled-transition}. There the argument requires convergence of products involving the rapidly varying phase field and \(\nabla Q_\varepsilon\), for which weak \(L^2\) convergence of \(Q_\varepsilon\) is insufficient. The truncated-\(SBV^2\) convergence supplies precisely the distributional compactness needed to identify these terms, without any additional \(H^1_{\mathrm{loc}}\) or \(L^\infty\) assumption on the diffuse tensor sequence.
\end{remark}

\begin{remark}\label{rem:SBV-convergence-notation}
Throughout the lower-bound argument, $V_\varepsilon^M\sbvto T_M(Q)$ means $V_\varepsilon^M\to T_M(Q)$ strongly in $L^1(\Omega;\mathcal S_0)$ and $\nabla V_\varepsilon^M\rightharpoonup\chi_A\nabla T_M(Q|_A)$ weakly in $L^2(\Omega;\mathcal S_0\otimes\mathbb R^3)$. Here $T_M(Q)$ denotes the zero extension of $T_M(Q|_A)$.
\end{remark}

\subsection{Tubular coordinates near the limiting interface}
\label{sec:tubular}
Let $(Q,A)\in\mathcal A_0^{\rm reg}$. Since $A\Subset\Omega$ has $C^2$ boundary, there exists $\rho_0>0$ such that the signed distance function
\[
d_A(x):=
\begin{cases}
-\operatorname{dist}(x,\partial A), & x\in A,\\
\operatorname{dist}(x,\partial A), & x\in A^c,
\end{cases}
\]
is $C^2$ on $U_{\rho_0}:=\{x\in\Omega:|d_A(x)|<\rho_0\}$. The normal-coordinate map
\begin{equation*}
    \Psi:\partial A\times(-\rho_0,\rho_0)\to U_{\rho_0},\qquad\Psi(y,t):=y+t\nu_A(y),
\end{equation*}
is a $C^1$-diffeomorphism. Set
\[
    J_A(y,t):=|\det D\Psi(y,t)|
\]
Uniformly for \(y\in\partial A\) and \(|t|<\rho_0\),
\[
    |J_A(y,t)-1|\leq C|t|.
\]
Consequently, for every integrable \(f\),
\begin{equation}\label{eq:tubular-change-of-variables}
    \int_{U_\rho}f(x)\,dx=\int_{\partial A}\int_{-\rho}^{\rho}f(\Psi(y,t))J_A(y,t)\,dt\,d\mathcal H^2(y).
\end{equation}
Moreover, if \(f\in H^1(U_{\rho_0})\), then
\begin{equation}\label{eq:normal-slice-gradient}
    |\partial_t(f\circ\Psi)(y,t)|\leq|\nabla f(\Psi(y,t))|.
\end{equation}
for \(\mathcal H^2\)-a.e. \(y\) and a.e. \(t\).

For $0<\rho<\rho_0$, set
\[
    U_\rho:=\{|d_A|<\rho\},\qquad U_\rho^{\rm int}:=U_\rho\cap A,\qquad U_\rho^{\rm ext}:=U_\rho\cap A^c,
\]
and, for $0<\delta,\rho<\rho_0$,
\[
    N_{\delta,\rho}:=U_\rho^{\rm ext}\cup U_\delta^{\rm int},\qquad A_\rho:=\{x\in A:\operatorname{dist}(x,\partial A)>\rho\}.
\]
Up to sets of measure zero,
\[
    U_\rho^{\rm int}=\Psi\bigl(\partial A\times(-\rho,0)\bigr),\qquad U_\rho^{\rm ext}=\Psi\bigl(\partial A\times(0,\rho)\bigr).
\]

Also
\begin{equation}\label{eq:tubular-jacobian-lower-bound}
    c_{\delta,\rho}:=\inf_{\partial A\times(-\delta,\rho)}J_A,
\end{equation}
where
\[
    1-C(\delta+\rho)\leq c_{\delta,\rho}\leq 1, \qquad c_{\delta,\rho}\to1\quad\text{as }\delta,\rho\downarrow0.
\]
The same notation and estimates are used componentwise for $\mathcal S_0$-valued maps.

\section{Lower Bound}\label{sec:lower-bound}
We prove the liminf inequality. The bulk, mixing, and relaxation contributions are treated separately, while the interaction of the two interfacial scales appears only in the latter two terms.

The relative scale $r_\eps=\kap/\eps$ determines how these interfacial contributions must be combined. When $r_\eps\to +\infty$, the phase transition and exterior relaxation can be estimated separately. When $r_\eps\to0$, the relaxation is strong enough to force the interfacial trace $Q^-$ to vanish. In the critical regime $r_\eps\to\ell\in(0,\infty)$, part of the void potential contributes within the phase-transition layer, and a coupled lower-bound argument is required.

Throughout this section, let $(Q_\eps,\phi_\eps)\in\mathcal A_\eps$ be an energy-bounded sequence such that $(Q_\eps,\phi_\eps)\to(Q,A)$, where $(Q,A)\in\mathcal A_0^{\rm reg}$. No additional interior compactness assumption is imposed on the diffuse tensor sequence $Q_\eps$.

\subsection{The bulk contribution}

The bulk term is expected to survive only in the interior region $A$, where $\phi_\eps\to1$. The coefficient $\phi_\eps^2$ may nevertheless degenerate on small subsets of $A$, so the diffuse energy does not provide an unweighted local $H^1$ bound for the original sequence $Q_\eps$.

We instead use Proposition~\ref{prop:truncated-sbv-compactness}. After restricting $Q_\eps$ to suitable high-phase sets and truncating its amplitude, the resulting tensors are weakly compact in $SBV^2$. This is sufficient to recover the absolutely continuous gradient of the regular limiting tensor $Q$ and to establish the bulk lower bound.

\begin{proposition}[Bulk lower bound]
\label{prop:bulk-lower-bound}
Let $(Q_\eps,\phi_\eps)\in\mathcal A_\eps$ be an energy-bounded sequence such that $(Q_\eps,\phi_\eps)\to(Q,A)$ with $(Q,A)\in\mathcal A_0^{\rm reg}$. Then
\begin{equation*}
    \liminf_{\eps\to0} E_\eps^{\rm bulk}\ge\int_A\left(|\nabla Q|^2+\frac1{2\lambda}F_b(Q)\right)dx.
\end{equation*}
Moreover, for every open set $A'\Subset A$,
\begin{equation*}
    \liminf_{\eps\to0} E_\eps^{\rm bulk}(A')\ge\int_{A'}\left(|\nabla Q|^2+\frac1{2\lambda}F_b(Q) \right)dx,
\end{equation*}
where
\[
    E_\eps^{\rm bulk}(U):=\int_U\phi_\eps^2\left(|\nabla Q_\eps|^2+\frac1{2\lambda}F_b(Q_\eps)\right)dx
\]
for every measurable set $U\subset\Omega$.
\end{proposition}
\begin{proof}
Fix an open set $A' \Subset A$. We first consider the gradient term. Passing to a subsequence realizing the corresponding lower limit, and then to a further subsequence if necessary, Proposition~\ref{prop:truncated-sbv-compactness} gives
\[
    V_{\varepsilon}^{M}\sbvto T_M(Q) \qquad\text{for every }M\in\mathbb N.
\]
Since $A'\Subset A$ and $Q\in H^1(A;\mathcal S_0)$, the absolutely continuous part of the derivative of the zero extension of $T_M(Q)$ coincides with $\nabla T_M(Q)$ on $A'$. Therefore,
\[
    \int_{A'} |\nabla T_M(Q)|^2\,dx\leq\liminf_{\varepsilon\to0}\int_{A'} |\nabla V_{\varepsilon}^{M}|^2\,dx.
\]
Moreover,
\[
    \nabla V_{\varepsilon}^{M}=\chi_{\Omega_{\varepsilon}^{+}}\nabla T_M(Q_{\varepsilon})\qquad\text{a.e. in }\Omega.
\]
Since $T_M$ is $1$-Lipschitz and $\phi_{\varepsilon}\geq t_{\varepsilon}\geq a$ on $\Omega_{\varepsilon}^{+}$, using \eqref{eq:truncation-gradient} we obtain
\begin{align*}
    \int_{A'} |\nabla V_{\varepsilon}^{M}|^2\,dx=\int_{A'\cap\Omega_{\varepsilon}^{+}}|\nabla T_M(Q_{\varepsilon})|^2\,dx\leq\int_{A'\cap\Omega_{\varepsilon}^{+}}|\nabla Q_{\varepsilon}|^2\,dx\leq \frac{1}{a^2}\int_{A'}\phi_{\varepsilon}^2|\nabla Q_{\varepsilon}|^2\,dx .
\end{align*}
Consequently,
\begin{align}\label{eq:localized-truncated-gradient-liminf}
    a^2\int_{A'}|\nabla T_M(Q)|^2\,dx\leq\liminf_{\varepsilon\to0}\int_{A'}\phi_{\varepsilon}^2|\nabla Q_{\varepsilon}|^2\,dx.
\end{align}
By \eqref{eq:truncation-gradient-error} and the pointwise convergence $\nabla T_M(Q)\to\nabla Q$ a.e. in $A$, dominated convergence allows us to let $M\to\infty$ in \eqref{eq:localized-truncated-gradient-liminf}. Letting subsequently $a\uparrow1$, we obtain
\begin{equation}\label{eq:localized-bulk-gradient-liminf}
    \liminf_{\varepsilon\to0}\int_{A'}\phi_{\varepsilon}^2|\nabla Q_{\varepsilon}|^2\,dx\geq\int_{A'}|\nabla Q|^2\,dx.
\end{equation}

We next treat the bulk potential. Taking a subsequence realizing the lower limit and using \eqref{eq:strong-l1-Q} and \eqref{eq:phase-field-L1}, we may assume pointwise convergence a.e. in $\Omega$. Since $A'\Subset A$ and $F_b$ is continuous,
\[
    \phi_\varepsilon^2 F_b(Q_\varepsilon)\to F_b(Q) \qquad\text{a.e. in }A'.
\]
Therefore, by Fatou's lemma,
\begin{equation}\label{eq:localized-bulk-potential-liminf}
    \liminf_{\varepsilon\to0}\int_{A'}\phi_\varepsilon^2 F_b(Q_\varepsilon)\,dx\geq\int_{A'}F_b(Q)\,dx.
\end{equation}

By the superadditivity of the $\liminf$, combining \eqref{eq:localized-bulk-gradient-liminf} and \eqref{eq:localized-bulk-potential-liminf} yields 
\begin{equation*}
    \liminf_{\varepsilon\to0}E_\varepsilon^{\mathrm{bulk}}(A')\geq\int_{A'}\left(|\nabla Q|^2+\frac{1}{2\lambda}F_b(Q)\right)\,dx.
\end{equation*}

Finally, let $\{A_j\}_{j\in\mathbb N}$ be an increasing sequence of open sets such that
\[
    A_j\Subset A,\qquad\bigcup_{j=1}^{\infty}A_j=A\quad\text{up to a null set}.
\]
Since $E_{\varepsilon}^{\mathrm{bulk}}\geq E_{\varepsilon}^{\mathrm{bulk}}(A_j)$, the localized estimate gives
\[
    \liminf_{\varepsilon\to0}E_{\varepsilon}^{\mathrm{bulk}}\geq\int_{A_j}\left(|\nabla Q|^2+\frac{1}{2\lambda}F_b(Q)\right)\,dx
\]
for every $j$. Letting $j\to\infty$ and using monotone convergence proves the result.
\end{proof}

\subsection{The anisotropic phase-transition contribution}\label{subsec:anisotropic}
The anisotropic mixing term requires identifying the distributional limit of $Q_\varepsilon D\Phi(\phi_\varepsilon)$, since the tensor coefficient is not strongly controlled across the transition layer. The next lemma combines the truncated-$SBV^2$ convergence with an integration-by-parts argument.

The main point is to identify the term \(\Phi(\phi_\varepsilon)\nabla Q_\varepsilon\). Although \(\nabla Q_\varepsilon\) is not locally weakly compact a priori, Proposition \ref{prop:truncated-sbv-compactness} allows us to perform this identification through the gradients of the truncated restrictions \(V_\varepsilon^M\).
\begin{lemma}[Transition measure]\label{lem:transition-measure}
Let $(Q_\varepsilon,\phi_\varepsilon)\in\mathcal A_\varepsilon$
be an energy-bounded sequence such that $(Q_\varepsilon,\phi_\varepsilon)\to(Q,A)$ with $(Q,A)\in\mathcal A_0^{\mathrm{reg}}$. Then
\begin{equation}\label{eq:tensor-transition-product}
    Q_\varepsilon D\Phi(\phi_\varepsilon) \to c_0Q^-D\chi_A\quad\text{in }\mathcal D'(\Omega;\mathcal S_0\otimes\mathbb R^3).
\end{equation}
Moreover, the measures $\mu_\varepsilon:= A(Q_\varepsilon)D\Phi(\phi_\varepsilon)$ are uniformly bounded in $\mathcal M(\Omega;\mathbb R^3)$ and satisfy
\begin{equation}\label{eq:anisotropic-transition-measure}
    \mu_\varepsilon\stackrel{*}{\rightharpoonup} c_0 A(Q^-)D\chi_A\quad\text{in }\mathcal M(\Omega;\mathbb R^3).
\end{equation}
Consequently,
\begin{equation}\label{eq:combined-transition-measure}
    \Lambda_\varepsilon:=\left(D\Phi(\phi_\varepsilon),\sqrt{\omega}\, A(Q_\varepsilon)D\Phi(\phi_\varepsilon)\right)\stackrel{*}{\rightharpoonup}\Lambda:=\left(c_0D\chi_A,\sqrt{\omega}\,c_0 A(Q^-)D\chi_A\right).
\end{equation}
\end{lemma}
\begin{proof}
Fix an arbitrary subsequence. By Proposition~\ref{prop:truncated-sbv-compactness} and Remark~\ref{rem:SBV-convergence-notation}, after passing to a further subsequence, not relabeled,
\[
    V_\varepsilon^M\sbvto T_M(Q)\qquad\text{for every }M\in\mathbb N.
\]
We prove all the asserted convergences along this subsequence.

We first claim that
\begin{equation}\label{eq:weighted-gradient-distribution}
    \Phi(\phi_\varepsilon)\nabla Q_\varepsilon\to c_0\chi_A\nabla Q\quad\text{in }\mathcal D'(\Omega;\mathcal S_0\otimes\mathbb R^3).
\end{equation}
Let $\Psi\in C_c^\infty(\Omega;\mathcal S_0\otimes\mathbb R^3)$. For fixed $M$, by the definition of $V_\varepsilon^M$,
\begin{align*}
    \int_\Omega\Phi(\phi_\varepsilon)\Psi:\nabla T_M(Q_\varepsilon)\,dx=\int_\Omega\Phi(\phi_\varepsilon)\Psi:\nabla V_\varepsilon^M\,dx+\int_{\Omega\setminus\Omega_\varepsilon^+}\Phi(\phi_\varepsilon)\Psi:\nabla T_M(Q_\varepsilon)\,dx.
\end{align*}

Since \(\Phi\) is bounded on \([0,1]\), the strong \(L^1\)-convergence in \eqref{eq:Phi-phase-L1} improves to strong \(L^2\)-convergence:
\[
    \|\Phi(\phi_\varepsilon)-c_0\chi_A\|_{L^2(\Omega)}^2\leq C\|\Phi(\phi_\varepsilon)-c_0\chi_A\|_{L^1(\Omega)}\to0.
\]
On the other hand,
\[
    \nabla V_\varepsilon^M\rightharpoonup\chi_A\nabla T_M(Q) \qquad\text{weakly in }L^2(\Omega).
\]
Therefore, for every \(\Psi\in C_c^\infty (\Omega;\mathcal S_0\otimes\mathbb R^3)\),
\[
    \int_\Omega\Phi(\phi_\varepsilon)\Psi:\nabla V_\varepsilon^M\,dx\to c_0\int_A\Psi:\nabla T_M(Q)\,dx.
\]

Using $0\leq\Phi(s)\leq s^2/2$ for $s\in[0,1]$ and \eqref{eq:truncation-gradient},
\begin{align*}
    &\left|\int_{\Omega\setminus\Omega_\varepsilon^+}\Phi(\phi_\varepsilon)\Psi:\nabla T_M(Q_\varepsilon)\,dx\right|\leq C\|\Psi\|_{L^\infty} \|\phi_\eps|\nabla Q_\eps|\|_{L^2(\Omega)} \left(\int_{\Omega\setminus\Omega_\varepsilon^+}\phi_\varepsilon^2\,dx\right)^{1/2}\to 0,
\end{align*}
using
\[
    \int_{\Omega\setminus\Omega_\varepsilon^+}\phi_\varepsilon^2\,dx\leq|A\setminus\Omega_\varepsilon^+|+\int_{A^c}\phi_\varepsilon\,dx\to0
\]
by \eqref{eq:high-phase-perimeter} and \eqref{eq:phase-field-L1}. Therefore
\[
    \int_\Omega\Phi(\phi_\varepsilon)\Psi:\nabla T_M(Q_\varepsilon)\,dx\to c_0\int_A\Psi:\nabla T_M(Q)\,dx.
\]

It remains to remove the truncation. By \eqref{eq:truncation-gradient-error} and \eqref{eq:weighted-tensor-L4},
\begin{align*}
    &\left|\int_\Omega\Phi(\phi_\varepsilon)\Psi:\bigl(\nabla Q_\varepsilon-\nabla T_M(Q_\varepsilon)\bigr)\,dx\right|\\
    &\quad\leq C\|\Psi\|_{L^\infty}\left(\int_\Omega\phi_\varepsilon^2|\nabla Q_\varepsilon|^2\,dx\right)^{1/2}\left(\int_{\{|Q_\varepsilon|\geq M\}}\phi_\varepsilon^2\,dx\right)^{1/2}
    \leq\frac{C}{M^2}\|\Psi\|_{L^\infty}.
\end{align*}
Furthermore, the properties of the radial truncation and \eqref{eq:truncation-gradient} give
\[
    \nabla T_M(Q)\to\nabla Q \qquad\text{strongly in }L^2(A)
\]
as $M\to\infty$. Letting first $\varepsilon\to0$ and then $M\to\infty$ proves \eqref{eq:weighted-gradient-distribution}.

Next, we claim that
\[
    \Phi(\phi_\varepsilon)Q_\varepsilon\to c_0\chi_AQ\qquad\text{strongly in }L^1(\Omega;\mathcal S_0).
\]
Indeed,
\begin{align*}
    \|\Phi(\phi_\varepsilon)Q_\varepsilon-c_0\chi_AQ\|_{L^1}&\leq\|\Phi(\phi_\varepsilon)(Q_\varepsilon-Q)\|_{L^1}+\|(\Phi(\phi_\varepsilon)-c_0\chi_A)Q\|_{L^1}\\
    &\leq\|\Phi\|_{L^\infty([0,1])}\|Q_\varepsilon-Q\|_{L^1}+\|\Phi(\phi_\varepsilon)-c_0\chi_A\|_{L^2}\|Q\|_{L^2},
\end{align*}
and both terms converge to zero.

For $\Psi\in C_c^1 (\Omega;\mathcal S_0\otimes\mathbb R^3)$, integration by parts gives
\begin{align*}
    \left\langle Q_\varepsilon D\Phi(\phi_\varepsilon),\Psi \right\rangle= -\int_\Omega\Phi(\phi_\varepsilon)Q_\varepsilon:\operatorname{div}\Psi\,dx-\int_\Omega\Phi(\phi_\varepsilon)\Psi:\nabla Q_\varepsilon\,dx.
\end{align*}
Passing to the limit using
\eqref{eq:weighted-gradient-distribution}, we obtain
\begin{align*}
    \lim_{\varepsilon\to0}\left\langle Q_\varepsilon D\Phi(\phi_\varepsilon),\Psi\right\rangle=-c_0\int_{\partial A}(Q^-\otimes\nu_A):\Psi\,d\mathcal H^2=\left\langle c_0Q^-D\chi_A,\Psi\right\rangle.
\end{align*}
This proves \eqref{eq:tensor-transition-product} along the selected subsequence.

Finally, the energy bound gives
\[
    |\mu_\varepsilon|(\Omega)\leq\left(\int_\Omega\varepsilon^{-1}W(\phi_\varepsilon)\,dx\right)^{1/2}\left(\int_\Omega\varepsilon\left|A(Q_\varepsilon)\nabla\phi_\varepsilon\right|^2\,dx\right)^{1/2}\leq C.
\]
From the definition of $A$, see \eqref{eq:anchoring-matrix}, \eqref{eq:phase-transition-measure} and \eqref{eq:tensor-transition-product} imply $\mu_\varepsilon\to c_0 A(Q^-)D\chi_A$ in $\mathcal D'(\Omega;\mathbb R^3).$ Since $Q^-\in L^2(\partial A;\mathcal S_0)$ and $\mathcal H^2(\partial A)<+\infty$, the limiting distribution $c_0 A(Q^-)D\chi_A$ is a finite Radon measure. Hence the uniform measure bound yields \eqref{eq:anisotropic-transition-measure}. Finally, \eqref{eq:combined-transition-measure} follows from \eqref{eq:phase-transition-measure} and \eqref{eq:anisotropic-transition-measure}.
\end{proof}

\begin{proposition}[Classical mixing lower bound]\label{prop:mixing-lower-bound}
Let $(Q_\varepsilon,\phi_\varepsilon)\in\mathcal A_\varepsilon$ be an energy-bounded sequence such that $(Q_\varepsilon,\phi_\varepsilon)\to(Q,A)$ with $(Q,A)\in\mathcal A_0^{\mathrm{reg}}$. Then
\[
    \liminf_{\varepsilon\to0}E_\varepsilon^{\mathrm{mix}} \geq\frac{w_p}{3} \int_{\partial A} G(Q^-,\nu_A)\,d\mathcal H^2.
\]
\end{proposition}

\begin{proof}
The arithmetic--geometric mean inequality and the definitions of $\Phi$, $F$, and $\Lambda_\varepsilon$ give
\[
    E_\varepsilon^{\mathrm{mix}}\geq2w_p \int_\Omega F\bigl(Q_\varepsilon,\nabla\Phi(\phi_\varepsilon)\bigr)\,dx= 2w_p|\Lambda_\varepsilon|(\Omega).
\]
Therefore, by \eqref{eq:combined-transition-measure} and the weak-$*$ lower semicontinuity of total variation,
\[
    \liminf_{\varepsilon\to0} E_\varepsilon^{\mathrm{mix}}\geq 2w_p|\Lambda|(\Omega)= 2w_pc_0\int_{\partial A}G(Q^-,\nu_A)\,d\mathcal H^2=\frac{w_p}{3}\int_{\partial A} G(Q^-,\nu_A)\,d\mathcal H^2.
\]
\end{proof}

\subsection{Exterior tensor relaxation}
The exterior-relaxation energy is estimated along normal slices of a tubular neighborhood of $\partial A$. Since small high-phase intrusions may destroy the pointwise coercivity of the void potential, we first establish a one-dimensional capacity estimate that is stable under such exceptional sets.
\begin{lemma}[One-dimensional capacity comparison]\label{lem:one-dimensional-capacity-comparison}
Let \(\delta>0\), \(0<\sigma<\delta\), \(\kappa>0\), \(c>0\), and \(\eta\in(0,1/2)\). Set $\lambda_\eta:=1-\eta$. Let \(H\subset(0,\delta)\) be a measurable set with $|H|\leq \sigma$. For \(\alpha_0>0\) and \(q_0\in\mathbb R^m\), define
\begin{align*}
    \mathcal C_H(q_0):=\inf_{\substack{q\in H^1(0,\delta;\mathbb R^m)\\ q(0)=q_0}}\Bigg\{\int_H \alpha_0 |q'|^2\,dt+c\int_{(0,\delta)\setminus H}\left(\kappa |q'|^2+\kappa^{-1}\lambda_\eta^2 |q|^2\right)\,dt\Bigg\}.
\end{align*}
Then $\mathcal C_H(q_0)\ge\mu_{\kappa,\eta,\sigma,\delta}|q_0|^2$, where
\begin{align*}
    \mu_{\kappa,\eta,\sigma,\delta}:=\frac{A_\sigma B_{\kappa,\eta,\sigma,\delta}}{A_\sigma+B_{\kappa,\eta,\sigma,\delta}},\qquad A_\sigma:=\frac{\alpha_0}{\sigma},\quad B_{\kappa,\eta,\sigma,\delta}:=c\lambda_\eta\tanh\left(\frac{\lambda_\eta(\delta-\sigma)}{\kappa}\right).
\end{align*}
Consequently,
\begin{align*}
    \lim_{\eta\downarrow0}\lim_{\sigma\downarrow0}\liminf_{\kappa\downarrow0}\mu_{\kappa,\eta,\sigma,\delta}=c.
\end{align*}
\end{lemma}

\begin{proof}
Since the functional is diagonal in the components of \(q\), it is enough to prove the result for scalar-valued \(q\). The vector-valued statement then follows by applying the scalar estimate to each component and summing.

We first prove the comparison for the case where \(H\) is a finite union of intervals. The general measurable case follows by approximation, as explained at the end of the proof.

For \(s\in[0,\delta]\) and \(z\in\mathbb R\), define the value function
\begin{align*}
    V_H(s,z):=\inf_{\substack{q\in H^1(s,\delta)\\ q(s)=z}}\int_s^\delta\left(a_H(t)|q'|^2+b_H(t)|q|^2\right)\,dt,
\end{align*}
where
\begin{align*}
    a_H(t):=
    \begin{cases}
        \alpha_0, & t\in H,\\
        c\kappa, & t\notin H,
    \end{cases}
    \qquad
    b_H(t):=
    \begin{cases}
        0, & t\in H,\\
        c\kappa^{-1}\lambda_\eta^2, & t\notin H.
    \end{cases}
\end{align*}
By homogeneity of the quadratic functional, there exists \(m_H(s)\ge0\) such that
\begin{align*}
    V_H(s,z)=m_H(s)z^2.
\end{align*}
In particular, $\mathcal C_H(q_0)=m_H(0)q_0^2$.

We now compute how \(m_H\) changes across homogeneous intervals. If an interval of length \(\ell\) is contained in \(H\), then on this interval the energy is only \(\alpha_0|q'|^2\). If the value at the right endpoint is \(m z^2\), then the effective left endpoint value is
\begin{align*}
    \mathsf H_\ell(m):=\frac{\alpha_0 m}{\alpha_0+\ell m}.
\end{align*}
Indeed,
\begin{align*}
    \inf_{z_1\in\mathbb R}\left\{\frac{\alpha_0}{\ell}|z_0-z_1|^2+m z_1^2\right\}=\frac{\alpha_0 m}{\alpha_0+\ell m}z_0^2.
\end{align*}

If an interval of length \(\ell\) is contained in \((0,\delta)\setminus H\), set $M:=c\lambda_\eta$, $\mu:=\frac{\lambda_\eta}{\kappa}$. On such an interval, the energy is
\begin{align*}
    c\int\left(\kappa |q'|^2+\kappa^{-1}\lambda_\eta^2 |q|^2\right)\,dt.
\end{align*}
If the value at the right endpoint is \(m z^2\), then the effective left endpoint value is
\begin{align*}
    \mathsf L_\ell(m):=M\,\frac{m+M\tanh(\mu\ell)}{M+m\tanh(\mu\ell)}.
\end{align*}
This follows by solving the Euler--Lagrange equation $q''=\mu^2 q$ with the natural boundary condition $c\kappa q'(\ell)+m q(\ell)=0$.

The key point is the following adjacent-exchange inequality: for every \(h,\ell>0\) and every \(m\ge0\),
\begin{align}\label{eq:adjacent-exchange}
    \mathsf H_h\bigl(\mathsf L_\ell(m)\bigr)\le\mathsf L_\ell\bigl(\mathsf H_h(m)\bigr).
\end{align}
Writing $\tau:=\tanh(\mu\ell)$, a direct computation gives
\begin{align*}
    &L_\ell(H_h(m))-H_h(L_\ell(m))\\
    &=\frac{M^2h\tau\bigl(M^2\alpha_0\tau+M^2hm\tau+2M\alpha_0m+Mhm^2+\alpha_0m^2\tau\bigr)}{\bigl(M\alpha_0+Mhm+\alpha_0m\tau\bigr)\bigl(M^2h\tau+M\alpha_0+Mhm+\alpha_0m\tau\bigr)}\geq0.
\end{align*}
Hence moving a high-phase interval to the left cannot increase the capacity.

Thus, if a low-phase interval lies immediately to the left of a high-phase interval, exchanging their order and moving the high-phase interval to the left cannot increase the capacity. Repeated application of \eqref{eq:adjacent-exchange} shows that, among all finite unions \(H\) with fixed measure \(h:=|H|\), the smallest capacity is attained when $H=(0,h)$. Therefore, $\mathcal C_H(q_0)\ge \mathcal C_{(0,h)}(q_0)$.

We now compute \(\mathcal C_{(0,h)}\). Let $z=q(h)$, on \((0,h)\),
\begin{align*}
    \int_0^h \alpha_0|q'|^2\,dt\ge\frac{\alpha_0}{h}|q_0-z|^2.
\end{align*}
On \((h,\delta)\), the standard one-dimensional estimate gives
\begin{align*}
\begin{aligned}
    c\inf_{\substack{q\in H^1(h,\delta)\\q(h)=z}}\int_h^\delta\left(\kappa |q'|^2+\kappa^{-1}\lambda_\eta^2 |q|^2\right)\,dt=c\lambda_\eta\tanh\left(
    \frac{\lambda_\eta(\delta-h)}{\kappa}\right)z^2.
\end{aligned}
\end{align*}
Hence
\begin{align*}
\begin{aligned}
    \mathcal C_{(0,h)}(q_0)&=\inf_{z\in\mathbb R}\left\{\frac{\alpha_0}{h}|q_0-z|^2+c\lambda_\eta\tanh\left(\frac{\lambda_\eta(\delta-h)}{\kappa}\right)z^2\right\}=\frac{A_hB_{\kappa,\eta,h,\delta}}{A_h+B_{\kappa,\eta,h,\delta}}q_0^2,
\end{aligned}
\end{align*}
where
\begin{align*}
    A_h:=\frac{\alpha_0}{h},\quad B_{\kappa,\eta,h,\delta}:=c\lambda_\eta\tanh\left(\frac{\lambda_\eta(\delta-h)}{\kappa}\right).
\end{align*}

It remains to compare \(h\le \sigma\). Define
\begin{align*}
    F(h):=\frac{A_hB_{\kappa,\eta,h,\delta}}{A_h+B_{\kappa,\eta,h,\delta}}=\frac{\alpha_0 c\lambda_\eta\tanh\left(\frac{\lambda_\eta(\delta-h)}{\kappa}\right)}{\alpha_0+h c\lambda_\eta\tanh\left(\frac{\lambda_\eta(\delta-h)}{\kappa}\right)}.
\end{align*}
Set
\[
    B(h):=c\lambda_\eta\tanh\left(\frac{\lambda_\eta(\delta-h)}{\kappa}\right),\qquad F(h)=\frac{\alpha_0B(h)}{\alpha_0+hB(h)}.
\]
Since $B^\prime(h)<0$,
\[
    F^\prime(h)=\frac{\alpha_0\bigl(\alpha_0B^\prime(h)-B(h)^2\bigr)}{\bigl(\alpha_0+hB(h)\bigr)^2}<0.
\]
Thus $F$ is decreasing on $(0,\delta)$. Since \(h=|H|\le \sigma\),
\begin{align*}
    \mathcal C_H(q_0)\ge F(h)q_0^2\ge F(\sigma)q_0^2=\mu_{\kappa,\eta,\sigma,\delta}|q_0|^2.
\end{align*}
The preceding argument proves the estimate when \(H\) is a finite union of intervals. For a general measurable set \(H\subset(0,\delta)\), choose finite unions of intervals \(H_n\) such that
\[
    \chi_{H_n}\to\chi_H\qquad\text{strongly in }L^1(0,\delta),
\]
and $|H_n|\le \sigma$ for every $n$. By Lemma~\ref{lem:capacity-stability},
\[
    C_{H_n}(q_0)\to C_H(q_0).
\]
Since
\[
    C_{H_n}(q_0)\ge\mu_{\kappa,\eta,\sigma,\delta}|q_0|^2\qquad\text{for every }n,
\]
passing to the limit yields
\[
    C_H(q_0)\ge\mu_{\kappa,\eta,\sigma,\delta}|q_0|^2.
\]

Finally,
\[
    B_{\kappa,\eta,\sigma,\delta}\to c(1-\eta)\qquad\text{as }\kappa\downarrow0,\qquad A_\sigma=\frac{\alpha_0}{\sigma}\to+\infty\qquad\text{as }\sigma\downarrow0.
\]
Therefore,
\[
    \lim_{\sigma\downarrow0}\liminf_{\kappa\downarrow0}\mu_{\kappa,\eta,\sigma,\delta}=c(1-\eta).
\]
Letting \(\eta\downarrow0\), we obtain
\[
    \lim_{\eta\downarrow0}\lim_{\sigma\downarrow0}\liminf_{\kappa\downarrow0}\mu_{\kappa,\eta,\sigma,\delta}=c.
\]
This completes the proof.
\end{proof}

We apply Lemma~\ref{lem:one-dimensional-capacity-comparison} on normal slices. Averaging over inner parallel surfaces avoids the need for pointwise trace convergence of $Q_\varepsilon$.

\begin{proposition}[Localized exterior-relaxation bound]\label{prop:localized-exterior-relaxation}
Let $0<\delta,\rho<\rho_0$, let $\eta\in(0,1/2)$, and let $\beta\in(0,1]$. Define
\begin{align*}
    &R_\eps^{\delta,\rho,\eta,\beta}:=\\
    &\quad\beta\int_{N_{\delta,\rho}\cap\{\phi_\eps>\eta\}}\phi_\eps^2|\nabla Q_\eps|^2\,dx+\frac{w_v}{2}\int_{N_{\delta,\rho}\cap\{\phi_\eps\le\eta\}}\left(\kappa_\eps|\nabla Q_\eps|^2+\kappa_\eps^{-1}(1-\phi_\eps)^2|Q_\eps|^2\right)dx.
\end{align*}
For $Q\in H^1(A;\mathcal S_0)$, set
\begin{equation}\label{eq:averaged-inner-trace}
    \mathcal B_\delta(Q):=\frac{2}{\delta}\int_{\delta/2}^{\delta}\int_{\partial A}|Q(\Psi(y,-s))|^2\,d\mathcal H^2(y)\,ds.
\end{equation}
Then, for every energy-bounded sequence $(Q_\eps,\phi_\eps)\to (Q,A)$ with $(Q,A)\in\mathcal A_0^{\rm reg}$, one has
\begin{equation}\label{eq:localized-exterior-relaxation}
    \liminf_{\eps\to0}R_\eps^{\delta,\rho,\eta,\beta}\ge c_{\delta,\rho}\,\mu_{\delta,\eta,\beta}\,\mathcal B_\delta(Q),
\end{equation}
where
\begin{equation}\label{eq:localized-relaxation-coefficient}
    \mu_{\delta,\eta,\beta}:=\frac{\displaystyle\frac{\beta\eta^2}{\delta}\frac{w_v}{2}(1-\eta)}{\displaystyle\frac{\beta\eta^2}{\delta}+\frac{w_v}{2}(1-\eta)}.
\end{equation}
Consequently,
\begin{equation}\label{eq:exterior-relaxation-trace}
    \lim_{\rho\downarrow0}\liminf_{\delta\downarrow0}\liminf_{\eps\to0}R_\eps^{\delta,\rho,\eta,\beta}\ge\frac{w_v}{2}(1-\eta)\int_{\partial A}|Q^-|^2\,d\mathcal H^2.
\end{equation}
In particular,
\begin{equation}\label{eq:exterior-relaxation-final}
    \lim_{\eta\downarrow0}\lim_{\rho\downarrow0}\liminf_{\delta\downarrow0}\liminf_{\eps\to0}R_\eps^{\delta,\rho,\eta,\beta}\ge\frac{w_v}{2}\int_{\partial A}|Q^-|^2\,d\mathcal H^2.
\end{equation}
\end{proposition}

\begin{proof}
Fix a subsequence realizing the lower limit in \eqref{eq:localized-exterior-relaxation}. By \eqref{eq:strong-l1-Q}, after passing to a further subsequence, not relabeled, all subsequent convergences hold along this sequence.

For $y\in\partial A$ and $s\in(\delta/2,\delta)$, write
\[
    q_{\varepsilon,y}(t) :=Q_\varepsilon(\Psi(y,t)),\qquad\phi_{\varepsilon,y}(t):=\phi_\varepsilon(\Psi(y,t)).
\]
Set
\[
    H_{\varepsilon,\eta}^+(y):=\left\{t\in(0,\rho):\phi_{\varepsilon,y}(t)>\eta\right\}.
\]
By \eqref{eq:phase-field-L1} and \eqref{eq:tubular-change-of-variables},
\begin{equation}\label{eq:exterior-high-phase-slices}
    \int_{\partial A}|H_{\varepsilon,\eta}^+(y)|\,d\mathcal H^2(y)\leq\frac{C}{\eta}\int_{U_\rho^{\mathrm{ext}}}\phi_\varepsilon\,dx\to0.
\end{equation}

Fix $a\in(0,\rho)$, and let $\Sigma_{\varepsilon,\eta,a}:=\left\{y\in\partial A:|H_{\varepsilon,\eta}^+(y)|\le a\right\}$. Then by \eqref{eq:exterior-high-phase-slices}, we have
\begin{equation}\label{eq:good-normal-set}
    \mathcal H^2\bigl(\partial A\setminus\Sigma_{\varepsilon,\eta,a}\bigr)\to0.
\end{equation}

Fix $M\in\mathbb N$. For $y\in\Sigma_{\varepsilon,\eta,a}$ and $s\in(\delta/2,\delta)$, set
\[
    H_{\varepsilon,\eta}^s(y):=\left\{t\in(-s,\rho):\phi_{\varepsilon,y}(t)>\eta\right\},\qquad q_{\varepsilon,y}^M:=T_M(q_{\varepsilon,y}).
\]
Then
\[
    |H_{\varepsilon,\eta}^s(y)|
\leq s+a\leq\delta+a.
\]

Define
\begin{align*}
    &\mathcal E_{\varepsilon,y}^s:=\\
    &\beta\int_{H_{\varepsilon,\eta}^s(y)}\phi_{\varepsilon,y}^2|q_{\varepsilon,y}'|^2\,dt+\frac{w_v}{2}\int_{(-s,\rho)\setminus H_{\varepsilon,\eta}^s(y)}\left(\kappa_\varepsilon|q_{\varepsilon,y}'|^2+\kappa_\varepsilon^{-1}(1-\phi_{\varepsilon,y})^2|q_{\varepsilon,y}|^2\right)\,dt .
\end{align*}

By \eqref{eq:truncation-gradient}, $|T_M(q)|\leq|q|$, and the definition of $H_{\varepsilon,\eta}^s(y)$,
\begin{align*}
    &\mathcal E_{\varepsilon,y}^s\geq{}\\
    &\beta\eta^2\int_{H_{\varepsilon,\eta}^s(y)}\left|(q_{\varepsilon,y}^M)'\right|^2\,dt+\frac{w_v}{2}\int_{(-s,\rho)\setminus H_{\varepsilon,\eta}^s(y)}\left(\kappa_\varepsilon\left|(q_{\varepsilon,y}^M)'\right|^2+\kappa_\varepsilon^{-1}(1-\eta)^2|q_{\varepsilon,y}^M|^2\right)\,dt.
\end{align*}
After translating $(-s,\rho)$ to $(0,s+\rho)$, Lemma~\ref{lem:one-dimensional-capacity-comparison} applies with $\alpha_0=\beta\eta^2$, $c=\frac{w_v}{2}$, $\sigma=s+a$ and $\delta=s+\rho$, it gives
\begin{equation}\label{eq:slice-capacity-bound}
    \mathcal E_{\varepsilon,y}^s\geq\mu_{\varepsilon,\delta,\rho,\eta,\beta,a}\left|T_M(q_{\varepsilon,y}(-s))\right|^2,
\end{equation}
where
\begin{align*}
    \mu_{\varepsilon,\delta,\rho,\eta,\beta,a}:=\frac{A_{\delta,a,\eta,\beta}B_{\varepsilon,\rho,a,\eta}}{A_{\delta,a,\eta,\beta}+B_{\varepsilon,\rho,a,\eta}}
\end{align*}
and
\begin{align*}
    A_{\delta,a,\eta,\beta}:=\frac{\beta\eta^2}{\delta+a},\quad B_{\varepsilon,\rho,a,\eta}:=\frac{w_v}{2}(1-\eta)\tanh\!\left(\frac{(1-\eta)(\rho-a)}
     {\kappa_\varepsilon}\right).
\end{align*}
Here we used $s+a\leq\delta+a$ and the monotonicity of
$AB/(A+B)$ in $A$.

By \eqref{eq:tubular-jacobian-lower-bound}, \eqref{eq:normal-slice-gradient}, and \eqref{eq:slice-capacity-bound}, for a.e. $s\in(\delta/2,\delta)$,
\[
    R_{\varepsilon}^{\delta,\rho,\eta,\beta}\geq c_{\delta,\rho}\mu_{\varepsilon,\delta,\rho,\eta,\beta,a}\int_{\Sigma_{\varepsilon,\eta,a}}\left|T_M(Q_\varepsilon(\Psi(y,-s)))\right|^2\,d\mathcal H^2(y).
\]
Averaging in $s$ gives
\begin{equation}\label{eq:averaged-slice-capacity-bound}
    R_{\varepsilon}^{\delta,\rho,\eta,\beta}\geq c_{\delta,\rho}\mu_{\varepsilon,\delta,\rho,\eta,\beta,a}\frac{2}{\delta}\int_{\delta/2}^{\delta}\int_{\Sigma_{\varepsilon,\eta,a}}\left|T_M(Q_\varepsilon(\Psi(y,-s)))\right|^2\,d\mathcal H^2(y)\,ds.
\end{equation}

By \eqref{eq:strong-l1-Q} and the Lipschitz continuity of \(T_M\),
\[
    T_M(Q_\varepsilon)\to T_M(Q) \qquad\text{strongly in }L^1(\Omega).
\]
Since both tensors are bounded in absolute value by \(M\),
\begin{align*}
    \|T_M(Q_\varepsilon)-T_M(Q)\|_{L^2(\Omega)}^2\leq2M\|T_M(Q_\varepsilon)-T_M(Q)\|_{L^1(\Omega)}\to0.
\end{align*}
Using the tubular change of variables and the uniform upper and lower bounds for \(J_A\), it follows that
\[
    T_M(Q_\varepsilon(\Psi(y,-s)))\to T_M(Q(\Psi(y,-s)))
\]
strongly in \(L^2(\partial A\times(\delta/2,\delta))\).

It remains to remove the varying exceptional set \(\partial A\setminus\Sigma_{\varepsilon,\eta,a}\). Since \(|T_M|\leq M\), we have
\begin{align*} 
    \frac{2}{\delta}\int_{\delta/2}^{\delta}\int_{\partial A\setminus\Sigma_{\varepsilon,\eta,a}}\left|T_M(Q_\varepsilon(\Psi(y,-s)))\right|^2\,d\mathcal H^2(y)\,ds\leq M^2\mathcal H^2\bigl(\partial A\setminus\Sigma_{\varepsilon,\eta,a}\bigr)\to0
\end{align*}
by \eqref{eq:good-normal-set}. The same estimate holds with \(Q_\varepsilon\) replaced by \(Q\). Consequently,
\begin{align}
    \frac{2}{\delta}\int_{\delta/2}^{\delta}\int_{\Sigma_{\varepsilon,\eta,a}}\left|T_M(Q_\varepsilon(\Psi(y,-s)))\right|^2\,d\mathcal H^2(y)\,ds\to B_{\delta,M}(Q),\label{eq:averaged-truncated-trace-limit}
\end{align}
where
\[
    B_{\delta,M}(Q):=\frac{2}{\delta}\int_{\delta/2}^{\delta}\int_{\partial A}|T_M(Q(\Psi(y,-s)))|^2\,d\mathcal H^2(y)\,ds.
\]

Since $B_{\varepsilon,\rho,a,\eta}\to w_v(1-\eta)/2$ as $\eps\to 0$, \eqref{eq:averaged-slice-capacity-bound} and \eqref{eq:averaged-truncated-trace-limit} yield
\begin{align*}
    \liminf_{\varepsilon\to0}R_{\varepsilon}^{\delta,\rho,\eta,\beta}\geq c_{\delta,\rho}\frac{\frac{\beta\eta^2}{\delta+a}\frac{w_v}{2}(1-\eta)}{\frac{\beta\eta^2}{\delta+a}+\frac{w_v}{2}(1-\eta)}B_{\delta,M}(Q).
\end{align*}
For fixed $\delta$, $\rho$, $\eta$, and $\beta$, we first let $\varepsilon\to0$, then $a\downarrow0$, and finally $M\to\infty$. By monotone convergence, this proves \eqref{eq:localized-exterior-relaxation}.

By the continuity of the interior trace in tubular coordinates,
\begin{equation*}
    B_\delta(Q)\to\int_{\partial A}|Q^-|^2\,d\mathcal H^2\qquad\text{as }\delta\downarrow0.
\end{equation*}
Together with $c_{\delta,\rho}=1+O(\delta+\rho)$ and \eqref{eq:localized-relaxation-coefficient}, we first let $\delta\downarrow0$ and then $\rho\downarrow0$, with $\eta$ and $\beta$ fixed, to obtain \eqref{eq:exterior-relaxation-trace}. Finally, letting $\eta\downarrow0$ gives \eqref{eq:exterior-relaxation-final}.
\end{proof}

The same competition becomes singular when $\kappa_\eps\ll\eps$. In this regime, the void relaxation acts on a scale shorter than the phase transition. The following estimate shows that any nonzero tensor trace would then produce unbounded energy.

\begin{lemma}[Vanishing interfacial trace in the fast-relaxation regime]
\label{lem:strong-anchoring}
Assume $\frac{\kap}{\eps}\to 0$. Then $Q^-=0$ $\mathcal H^2-$ a.e. on $\partial A$. 
\end{lemma}
\begin{proof}
We have
\begin{align*}
    |Q_\varepsilon D\Phi(\phi_\eps)|(\Omega)&=\int_\Omega|Q_\varepsilon|\phi_\varepsilon(1-\phi_\varepsilon)|\nabla\phi_\varepsilon|\,dx \\
    &\le\left(\int_\Omega\kappa_\varepsilon^{-1}(1-\phi_\varepsilon)^2|Q_\varepsilon|^2\,dx\right)^{1/2}\left(\int_\Omega\kappa_\varepsilon\phi_\varepsilon^2|\nabla\phi_\varepsilon|^2\,dx\right)^{1/2}.
\end{align*}
The first factor is bounded by the void energy, while the second is at most \(C(\kappa_\varepsilon/\varepsilon)^{1/2}\). Thus \(Q_\varepsilon D\Phi(\phi_\varepsilon)\to0\) in total variation. Combining this with \eqref{eq:tensor-transition-product} gives $c_0Q^-D\chi_A=0$, and hence $Q^-=0$ $\mathcal H^2$-a.e. on $\partial A$.
\end{proof}

\subsection{Coupled transition in the critical regime}\label{subsec:coupled-transition}
Assume that
\[
r_\varepsilon:=\frac{\kappa_\varepsilon}{\varepsilon}
\to\ell\in(0,\infty).
\]
Since the phase transition and exterior relaxation occur on comparable scales, their contributions must be estimated jointly. We therefore introduce a $Q_\varepsilon$-dependent Modica primitive.

Define
\begin{equation}\label{eq:critical-effective-potential}
    V_{\varepsilon,q}(s):=w_pW(s)+\frac{w_v}{2r_\varepsilon}|q|^2(1-s)^2,\qquad V_{\ell,q}(s):=w_pW(s)+\frac{w_v}{2\ell}|q|^2(1-s)^2.
\end{equation}
For fixed \(\eta\in(0,1)\), set
\begin{equation}\label{eq:critical-weighted-primitive}
    H_{\varepsilon,\eta}(q,s):=\int_\eta^{\max\{s,\eta\}}\sqrt{V_{\varepsilon,q}(t)}\,dt,\qquad c_{\ell,\eta}(q):=\int_\eta^1\sqrt{V_{\ell,q}(s)}\,ds.
\end{equation}
We also define
\begin{equation}\label{eq:critical-truncated-density}
    m_{\ell}^{\eta}(q,\nu):=2\sqrt{w_p}\, c_{\ell,\eta}(q)G(q,\nu).
\end{equation}

The next lemma is the critical analogue of Lemma~\ref{lem:transition-measure}. The main additional difficulty is that the phase-field primitive now depends on $Q_\eps$. Differentiating it therefore produces a term involving $\nabla Q_\eps$. This is the second, and more delicate, use of the truncated-\(SBV^2\) framework. Since the Modica-type primitive depends on \(Q_\varepsilon\), its derivative contains an additional \(\nabla Q_\varepsilon\)-term. Proposition \ref{prop:truncated-sbv-compactness} allows us to identify this term distributionally without imposing any additional local \(H^1\) compactness or uniform \(L^\infty\) bound on \(Q_\varepsilon\).

\begin{lemma}[Weighted transition measure in the critical regime]
\label{lem:critical-weighted-transition}
Assume that $r_\varepsilon\to\ell\in(0,\infty)$, $(Q_\varepsilon,\phi_\varepsilon)\to(Q,A)$ with $(Q,A)\in\mathcal A_0^{\mathrm{reg}}$. Fix $\eta,\theta\in(0,1)$, and set
\[
    h_\varepsilon:=H_{\varepsilon,\eta}(Q_\varepsilon,\phi_\varepsilon),\qquad g_\eta:=c_{\ell,\eta}(Q)\chi_A.
\]
For every open set $U\subset\Omega$, define
\begin{equation*}
\begin{aligned}
    \widetilde{\mathcal{T}}_{\varepsilon,\eta,\theta}(U):=&\int_{U\cap\{\phi_\varepsilon>\eta\}}\Bigg(\theta\phi_\varepsilon^2|\nabla Q_\varepsilon|^2+\frac{\phi_\varepsilon^2}{2\lambda}F_b(Q_\varepsilon)\\
    &\quad+\frac{w_p}{\varepsilon}W(\phi_\varepsilon)+w_p\varepsilon F(Q_\varepsilon,\nabla\phi_\varepsilon)^2+\frac{w_v}{2\kappa_\varepsilon}(1-\phi_\varepsilon)^2|Q_\varepsilon|^2\Bigg)\,dx.
\end{aligned}
\end{equation*}

Then
\begin{equation}\label{eq:critical-primitive-L1}
    h_\varepsilon\to g_\eta \qquad\text{strongly in }L^1(\Omega).
\end{equation}

Moreover, the $\mathbb R^3\times\mathbb R^3$-valued Radon measures $\Lambda_{\varepsilon,\eta}:=\left(Dh_\varepsilon,\sqrt{\omega}\, A(Q_\varepsilon)Dh_\varepsilon\right)$ are uniformly bounded in total variation and satisfy
\begin{equation}\label{eq:critical-measure-convergence}
    \Lambda_{\varepsilon,\eta}\stackrel{*}{\rightharpoonup}\Lambda_\eta:=\left(Dg_\eta,\sqrt{\omega}\,A(Q)Dg_\eta\right)
\end{equation}
in $\mathcal M\bigl(\Omega;\mathbb R^3\times\mathbb R^3\bigr).$

The limiting measures admit the decompositions
\begin{align}
    Dg_\eta&=\chi_A\nabla(c_{\ell,\eta}(Q))\,dx-c_{\ell,\eta}(Q^-)\nu_A\,\mathcal H^2\mres\partial A,\label{eq:critical-Dg-decomposition}\\
    A(Q)Dg_\eta&=\chi_AA(Q)\nabla(c_{\ell,\eta}(Q))\,dx-c_{\ell,\eta}(Q^-)A(Q^-)\nu_A\,\mathcal H^2\mres\partial A.\label{eq:critical-ADg-decomposition}
\end{align}

Finally, for every open set $U\subset\Omega$,
\begin{equation}\label{eq:critical-energy-measure-bound}
    \widetilde{\mathcal{T}}_{\varepsilon,\eta,\theta}(U)\geq2\sqrt{w_p}\,|\Lambda_{\varepsilon,\eta}|(U)-C_{\eta,\theta}|U|,
\end{equation}
and consequently
\begin{equation}\label{eq:critical-open-set-liminf}
    \liminf_{\varepsilon\to0}\widetilde{\mathcal{T}}_{\varepsilon,\eta,\theta}(U)\geq2\sqrt{w_p}\,|\Lambda_\eta|(U)-C_{\eta,\theta}|U|.
\end{equation}
\end{lemma}

\begin{proof}
Fix an arbitrary subsequence and choose $0<a<b<\eta$. By Proposition~\ref{prop:truncated-sbv-compactness} and Remark~\ref{rem:SBV-convergence-notation}, after passing to a further subsequence, not relabeled,
\eqref{eq:strong-l1-Q} holds and
\[
    V_\varepsilon^M\xrightarrow{\mathrm{SBV}}T_M(Q)\qquad\text{for every }M\in\mathbb N.
\]
Since $t_\varepsilon\in(a,b)$ and $b<\eta$, then $S_{\varepsilon,\eta} =\{\phi_\varepsilon>\eta\}\subset\{\phi_\varepsilon>t_\varepsilon\}=\Omega_\varepsilon^+.$ We prove the assertions along this subsequence.

We divide the proof into five steps.

\medskip
\noindent\emph{Step 1: Estimates for the weighted primitive.}
Since $r_\varepsilon\to\ell\in(0,\infty)$, the coefficients $r_\varepsilon^{-1}$ are uniformly bounded above and below. From \eqref{eq:critical-effective-potential}--\eqref{eq:critical-weighted-primitive}, one obtains
\begin{align}
    0\leq H_{\varepsilon,\eta}(q,s)&\leq C\mathbf 1_{\{s>\eta\}}(1+|q|),\label{eq:critical-primitive-growth}\\
    |H_{\varepsilon,\eta}(q,s)-H_{\varepsilon,\eta}(\widetilde q,s)|&\leq C|q-\widetilde q|,\label{eq:critical-primitive-q-Lipschitz}\\
    |H_{\varepsilon,\eta}(q,s)-H_{\varepsilon,\eta}(q,\widetilde s)|&\leq C(1+|q|)|s-\widetilde s|,\label{eq:critical-primitive-s-Lipschitz}\\
    |H_{\varepsilon,\eta}(q,s)-H_{\ell,\eta}(q,s)|&\leq C\left|\frac1{r_\varepsilon}-\frac1\ell\right|^{1/2}|q|.\label{eq:critical-primitive-scale-error}
\end{align}
To justify the chain rule used below, we record more explicitly the regularity of the weighted primitive. Set $c_\varepsilon:=\frac{w_v}{2r_\varepsilon}$. Since $r_\varepsilon\to\ell\in(0,+\infty)$, after discarding finitely many indices there exist constants $0<c_*<c^*<+\infty$ such that $c_*\leq c_\eps\leq c^*$. For $t\in[\eta,1]$, the identity $W(t)=t^2(1-t)^2$ gives
\[
    V_{\varepsilon,q}(t)
    =(1-t)^2\bigl(w_p t^2+c_\varepsilon|q|^2\bigr),\qquad \sqrt{V_{\varepsilon,q}(t)}=(1-t)\sqrt{w_p t^2+c_\varepsilon|q|^2}.
\]
In particular, since $t\ge\eta>0$, the integrand is continuously differentiable with respect to $q$, including at $q=0$.

For $s>\eta$, we have
\[
    D_sH_{\varepsilon,\eta}(q,s)=\sqrt{V_{\varepsilon,q}(s)},\qquad D_qH_{\varepsilon,\eta}(q,s)=\int_\eta^s\frac{c_\varepsilon(1-t)q}{\sqrt{w_p t^2+c_\varepsilon|q|^2}}\,dt.
\]
For $s<\eta$, both derivatives vanish. Thus $H_{\varepsilon,\eta}:\mathcal S_0\times[0,1]\longrightarrow[0,+\infty)$ is locally Lipschitz, although it need not be differentiable with respect to $s$ at $s=\eta$. The value assigned to $D_sH_{\varepsilon,\eta}$ at $s=\eta$ is immaterial for the Sobolev chain rule, since the Stampacchia property gives $\nabla \phi_\eps=0$ a.e. on $\{\phi_\varepsilon=\eta\}$.

The preceding formulas also imply
\begin{align*}
    |D_qH_{\varepsilon,\eta}(q,s)|&\le\sqrt{c^*}\int_\eta^s(1-t)\,dt\le C(s-\eta)_+\le Cs\,\mathbf 1_{\{s>\eta\}},
\end{align*}
and
\[
    |D_sH_{\varepsilon,\eta}(q,s)|\le C(1+|q|)\mathbf 1_{\{s>\eta\}}.
\]
These estimates imply \eqref{eq:critical-primitive-growth}--\eqref{eq:critical-primitive-s-Lipschitz}. Estimate \eqref{eq:critical-primitive-scale-error} follows from $|\sqrt{a}-\sqrt{b}|\le \sqrt{|a-b|}$ and the convergence $r_\varepsilon^{-1}\to\ell^{-1}$. We shall also use the corresponding Lipschitz estimate for the limiting primitive. Since
\[
    \left|D_q\left( (1-s)\sqrt{w_ps^2+\frac{w_v}{2\ell}|q|^2}\right)\right|\le C_\ell(1-s),
\]
the mean-value theorem gives
\begin{equation}
    \left|c_{\ell,\eta}(q)-c_{\ell,\eta}(\widetilde q)\right|\le C_{\ell,\eta}|q-\widetilde q|.
    \label{eq:critical-c-Lipschitz}
\end{equation}

\medskip
\noindent
\textit{Step 2: Chain rule.}
For $R>0$, define $h_\varepsilon^R:=H_{\varepsilon,\eta}\bigl(T_R(Q_\varepsilon),\phi_\varepsilon\bigr)$. Since $T_R(Q_\varepsilon)$ is bounded and $H_{\varepsilon,\eta}$ is Lipschitz on bounded subsets of $\mathcal S_0\times[0,1]$, the standard Sobolev chain rule yields
\begin{align*}
    \nabla h_\varepsilon^R=\mathbf 1_{\{\phi_\varepsilon>\eta\}}\sqrt{V_{\varepsilon,T_R(Q_\varepsilon)}(\phi_\varepsilon)}\,\nabla\phi_\varepsilon+D_qH_{\varepsilon,\eta}\bigl(T_R(Q_\varepsilon),\phi_\varepsilon\bigr):\nabla T_R(Q_\varepsilon)
\end{align*}
almost everywhere in $\Omega$.

We next let $R\to+\infty$. Pointwise almost everywhere,
\[
    T_R(Q_\varepsilon)\to Q_\varepsilon,\qquad \nabla T_R(Q_\varepsilon)\to\nabla Q_\varepsilon.
\]
Moreover,
\begin{align*}
    \left|D_qH_{\varepsilon,\eta}\bigl(T_R(Q_\varepsilon),\phi_\varepsilon\bigr):\nabla T_R(Q_\varepsilon)\right| \le C\phi_\varepsilon|\nabla Q_\varepsilon|,\\ \left|D_sH_{\varepsilon,\eta}\bigl(T_R(Q_\varepsilon),\phi_\varepsilon\bigr)\nabla\phi_\varepsilon\right|\le C(1+|Q_\varepsilon|)|\nabla\phi_\varepsilon|.
\end{align*}
For every fixed $\varepsilon$, both right-hand sides belong to $L^1(\Omega)$. Dominated convergence therefore gives $h_\varepsilon^R\to h_\varepsilon$ strongly in $W^{1,1}(\Omega)$. Consequently,
\begin{equation}\label{eq:critical-chain-rule}
    \nabla h_\varepsilon=\mathbf 1_{\{\phi_\varepsilon>\eta\}}\left(\sqrt{V_{\varepsilon,Q_\varepsilon}(\phi_\varepsilon)}\,\nabla\phi_\varepsilon+B_\varepsilon\right),\quad (B_\varepsilon)_j:=D_qH_{\varepsilon,\eta}(Q_\varepsilon,\phi_\varepsilon):\partial_jQ_\varepsilon.
\end{equation}
The preceding estimate for $D_qH_{\varepsilon,\eta}$ gives
\begin{equation}\label{eq:critical-chain-error}
    |B_\varepsilon|\le C\phi_\varepsilon|\nabla Q_\varepsilon|\qquad\text{a.e. on }\{\phi_\varepsilon>\eta\}.
\end{equation}

Taking $s=1$ in the analogue of \eqref{eq:critical-primitive-q-Lipschitz} gives \eqref{eq:critical-c-Lipschitz}.

\medskip
\noindent\emph{Step 3: Strong convergence of $h_\varepsilon$.}
By \eqref{eq:critical-primitive-q-Lipschitz}--\eqref{eq:critical-primitive-scale-error},\eqref{eq:strong-l1-Q}, and \eqref{eq:phase-field-L1},
\begin{align*}
    \|h_\varepsilon-g_\eta\|_{L^1}\leq&C\|Q_\varepsilon-Q\|_{L^1}+C\left|\frac1{r_\varepsilon}-\frac1\ell\right|^{1/2}\|Q\|_{L^1}+C\int_\Omega(1+|Q|)|\phi_\varepsilon-\chi_A|\,dx
\end{align*}
converges to zero. This proves \eqref{eq:critical-primitive-L1} along the selected subsequence. For completeness,
set $d_\varepsilon:=|\phi_\varepsilon-\chi_A|$. Since $0\le d_\varepsilon\le1$, we have
\[
    \|d_\varepsilon\|_{L^2(\Omega)}^2\le\|d_\varepsilon\|_{L^1(\Omega)}\to0.
\]
Consequently,
\begin{align*}
    \int_\Omega(1+|Q|)d_\varepsilon\,dx\le\|d_\varepsilon\|_{L^1(\Omega)}+\|Q\|_{L^2(\Omega)}\|d_\varepsilon\|_{L^2(\Omega)}\to0,
\end{align*}
which justifies the last term in the estimate leading to \eqref{eq:critical-primitive-L1}.

We shall also need strong $L^2$ convergence. Set $S_{\varepsilon,\eta}:=\{\phi_\varepsilon>\eta\}$. By the energy bound, \eqref{eq:weighted-tensor-L4}, and $\phi_\varepsilon\ge\eta$ on $S_{\varepsilon,\eta}$, we have
\begin{equation}\label{eq:critical-high-phase-bound}
    \sup_{\varepsilon>0}\int_{S_{\varepsilon,\eta}}\left(|\nabla Q_\varepsilon|^2+|Q_\varepsilon|^4\right)\,dx<+\infty.
\end{equation}
By \eqref{eq:critical-primitive-growth} and \eqref{eq:critical-high-phase-bound}, the sequence $\{h_\varepsilon\}_\varepsilon$ is bounded in $L^4(\Omega)$. Since $c_{\ell,\eta}$ has at most linear growth and $Q\in H^1(A)\hookrightarrow L^6(A)$, we also have $g_\eta\in L^4(\Omega)$.

By interpolation between $L^1$ and $L^4$,
\begin{align*}
    \|h_\varepsilon-g_\eta\|_{L^2(\Omega)}\le\|h_\varepsilon-g_\eta\|_{L^1(\Omega)}^{1/3}\|h_\varepsilon-g_\eta\|_{L^4(\Omega)}^{2/3}\to 0.
\end{align*}
Therefore,
\begin{equation}\label{eq:critical-primitive-L2}
    h_\varepsilon\to g_\eta\qquad\text{strongly in }L^2(\Omega).
\end{equation}

\medskip
\noindent\emph{Step 4: Energy estimate and BV compactness.}
By \eqref{eq:critical-chain-error}, \eqref{eq:anisotropic-norm-growth}, and \eqref{eq:bulk-potential-coercivity}, on $S_{\varepsilon,\eta}$, we have
\begin{equation}\label{eq:critical-chain-error-absorption}
    2\sqrt{w_p}\, F(Q_\varepsilon,B_\varepsilon)\leq\theta\phi_\varepsilon^2|\nabla Q_\varepsilon|^2+\frac{\phi_\varepsilon^2}{2\lambda}F_b(Q_\varepsilon)+C_{\eta,\theta}.
\end{equation}

The arithmetic--geometric mean inequality, \eqref{eq:critical-chain-rule}, the triangle inequality for $F$, and \eqref{eq:critical-chain-error-absorption} give
\begin{align*}
    &\theta\phi_\varepsilon^2|\nabla Q_\varepsilon|^2+\frac{\phi_\varepsilon^2}{2\lambda}F_b(Q_\varepsilon)+\varepsilon^{-1}V_{\varepsilon,Q_\varepsilon}(\phi_\varepsilon)+w_p\varepsilon F(Q_\varepsilon,\nabla\phi_\varepsilon)^2\\
    &\quad\geq2\sqrt{w_p}\,F(Q_\varepsilon,\nabla h_\varepsilon)-C_{\eta,\theta}
\end{align*}
a.e. on $S_{\varepsilon,\eta}$.

We regard $\Lambda_{\varepsilon,\eta}$ as an $\mathbb R^3\times\mathbb R^3$-valued Radon measure. Since $h_\varepsilon\in W^{1,1}(\Omega)$, this measure is initially absolutely continuous, and, for every open set $U\subset\Omega$,
\begin{align*}
    |\Lambda_{\varepsilon,\eta}|(U)&=\int_U\sqrt{|\nabla h_\varepsilon|^2+\omega|A(Q_\varepsilon)\nabla h_\varepsilon|^2}\,dx=\int_UF(Q_\varepsilon,\nabla h_\varepsilon)\,dx.
\end{align*}
Integration therefore gives \eqref{eq:critical-energy-measure-bound}.

Taking $U=\Omega$ in \eqref{eq:critical-energy-measure-bound} and using the energy bound, we obtain $\sup_{\varepsilon>0}|\Lambda_{\varepsilon,\eta}|(\Omega)<+\infty$. In particular,
\[
    \sup_{\varepsilon>0}
    |Dh_\varepsilon|(\Omega)
    <+\infty.
\]
Together with \eqref{eq:critical-primitive-L1}, the $BV$ compactness theorem and uniqueness of the distributional derivative give
\begin{equation}\label{eq:critical-Dh-convergence}
    Dh_\varepsilon\stackrel{*}{\rightharpoonup}Dg_\eta\qquad\text{in }\mathcal M(\Omega;\mathbb R^3).
\end{equation}

\medskip
\noindent\emph{Step 5: Identification of the anisotropic measure.}
We first claim that
\begin{equation}\label{eq:critical-weighted-gradient-limit}
    h_\varepsilon\nabla Q_\varepsilon\to g_\eta\nabla Q \quad\text{in }\mathcal D'(\Omega;\mathcal S_0\otimes\mathbb R^3).
\end{equation}

Since $h_\varepsilon=0$ outside $S_{\varepsilon,\eta}\subset\Omega_\varepsilon^+$,
\[
    \int_\Omega h_\varepsilon\Psi:\nabla T_M(Q_\varepsilon)\,dx=\int_\Omega h_\varepsilon\Psi:\nabla V_\varepsilon^M\,dx.
\]

By \eqref{eq:critical-primitive-L2}, $h_\varepsilon\to g_\eta$ strongly in $L^2(\Omega)$, while
\[
    \nabla V_\varepsilon^M\rightharpoonup\chi_A\nabla T_M(Q)\qquad\text{weakly in }L^2(\Omega).
\]
Since \(\{\nabla V_\varepsilon^M\}_\varepsilon\) is bounded in \(L^2\), we obtain
\begin{align*}
    \left|\int_\Omega(h_\varepsilon-g_\eta)\Psi:\nabla V_\varepsilon^M\,dx\right|\leq\|h_\varepsilon-g_\eta\|_{L^2}\|\Psi\|_{L^\infty}\|\nabla V_\varepsilon^M\|_{L^2}\to0.
\end{align*}
The weak \(L^2\)-convergence of \(\nabla V_\varepsilon^M\) then gives
\[
    \int_\Omega h_\varepsilon\Psi:\nabla V_\varepsilon^M\,dx\to\int_Ag_\eta\Psi:\nabla T_M(Q)\,dx.
\]

By \eqref{eq:truncation-gradient-error}, \eqref{eq:critical-primitive-growth}, and \eqref{eq:critical-high-phase-bound},
\[
    \sup_{\varepsilon>0}\left|\int_\Omega h_\varepsilon\Psi:\bigl(\nabla Q_\varepsilon-\nabla T_M(Q_\varepsilon)\bigr)\,dx\right|\leq\frac{C_\eta}{M}\|\Psi\|_{L^\infty}.
\]
Letting first $\varepsilon\to0$ and then $M\to\infty$, using \eqref{eq:truncation-gradient}, proves \eqref{eq:critical-weighted-gradient-limit}.

Furthermore, $h_\varepsilon A(Q_\varepsilon)\to g_\eta A(Q)$ strongly in $L^1(\Omega)$. Indeed, the convergence holds in measure by \eqref{eq:strong-l1-Q} and \eqref{eq:critical-primitive-L1}. Moreover, by \eqref{eq:critical-primitive-growth},
\[
    |h_\varepsilon A(Q_\varepsilon)|^2\le C_\eta(1+|Q_\varepsilon|^4)\mathbf 1_{S_{\varepsilon,\eta}}.
\]
It follows from \eqref{eq:critical-high-phase-bound} that $\{h_\varepsilon A(Q_\varepsilon)\}_\varepsilon$ is bounded in $L^2(\Omega)$ and is therefore uniformly integrable in $L^1(\Omega)$. The asserted strong $L^1$ convergence now follows from Vitali's theorem.

For $\zeta\in C_c^1(\Omega;\mathbb R^3)$, integration by parts gives
\begin{align*}
    \int_\Omega\zeta\cdot A(Q_\varepsilon)\nabla h_\varepsilon\,dx=-\int_\Omega h_\varepsilon A(Q_\varepsilon):\nabla\zeta\,dx-\int_\Omega h_\varepsilon\operatorname{div}Q_\varepsilon\cdot\zeta\,dx,
\end{align*}
where we used the definition of $A$, see \eqref{eq:anchoring-matrix}. The first term passes to the limit by the strong $L^1$ convergence above, while the second passes to the limit by \eqref{eq:critical-weighted-gradient-limit}. Therefore,
\[
    A(Q_\varepsilon)Dh_\varepsilon\to A(Q)Dg_\eta\qquad\text{in }\mathcal D'(\Omega;\mathbb R^3).
\]

We next specify the meaning of the measure on the right-hand side. By \eqref{eq:critical-c-Lipschitz}, the map $c_{\ell,\eta}:\mathcal S_0\to[0,+\infty)$ is globally Lipschitz and has at most linear growth. Since $Q\in H^1(A;\mathcal S_0)$, the Sobolev chain rule gives $c_{\ell,\eta}(Q)\in H^1(A)$. For $g_\eta=c_{\ell,\eta}(Q)\chi_A$, the $BV$ product rule gives
\begin{align*}
    Dg_\eta=\chi_A\nabla(c_{\ell,\eta}(Q))\,dx-c_{\ell,\eta}(Q^-)\nu_A\,\mathcal H^2\!\llcorner\partial A.
\end{align*}
We define
\begin{align*}
    A(Q)Dg_\eta:=\chi_AA(Q)\nabla(c_{\ell,\eta}(Q))\,dx-c_{\ell,\eta}(Q^-)A(Q^-)\nu_A\,\mathcal H^2\!\llcorner\partial A.
\end{align*}
This is a finite Radon measure. Indeed, $A(Q)\in L^6(A)$ and $\nabla(c_{\ell,\eta}(Q))\in L^2(A)$, so that
\[
    A(Q)\nabla(c_{\ell,\eta}(Q))\in L^{3/2}(A)\subset L^1(A).
\]
Moreover,
\[
    c_{\ell,\eta}(q)|A(q)|\le C_\eta(1+|q|^2),
\]
and $Q^-\in L^2(\partial A;\mathcal S_0)$. Hence the boundary density belongs to $L^1(\partial A)$.

The uniform bound on $|\Lambda_{\varepsilon,\eta}|(\Omega)$, together with the preceding distributional identification, now yields
\[
    \Lambda_{\varepsilon,\eta}\stackrel{*}{\rightharpoonup}\left(Dg_\eta,\sqrt{\omega}\,A(Q)Dg_\eta\right)
\]
in $\mathcal M\bigl(\Omega;\mathbb R^3\times\mathbb R^3\bigr)$. This proves \eqref{eq:critical-measure-convergence}, while the preceding formulas give \eqref{eq:critical-Dg-decomposition} and \eqref{eq:critical-ADg-decomposition}.

We have proved the assertions along a further subsequence of an arbitrary initial subsequence. Since all possible subsequential limits are the same, namely $g_\eta$ in \eqref{eq:critical-primitive-L1} and $\Lambda_\eta$ in \eqref{eq:critical-measure-convergence}, the subsequence criterion implies that \eqref{eq:critical-primitive-L1} and \eqref{eq:critical-measure-convergence} hold for the full sequence.
\end{proof}

\begin{proposition}[Critical interfacial lower bound]\label{prop:critical-transition-lower-bound}
Assume that $r_\eps\to\ell\in(0,\infty)$. For every $\eta,\theta\in(0,1)$ and $0<\delta,\rho<\rho_0$,
\begin{align}\label{eq:critical-interface-fixed-eta}
    \liminf_{\eps\to0}\widetilde{\mathcal{T}}_{\eps,\eta,\theta}(N_{\delta,\rho})\geq\int_{\partial A}m_{\ell}^{\eta}(Q^-,\nu_A)\,d\mathcal H^2-C_{\eta,\theta}|N_{\delta,\rho}|,
\end{align}
Consequently,
\[
    \lim_{\eta\downarrow0}\lim_{\delta,\rho\downarrow0}\liminf_{\eps\to0}\widetilde{\mathcal{T}}_{\eps,\eta,\theta}(N_{\delta,\rho})\geq\int_{\partial A}m_\ell(Q^-,\nu_A)\,d\mathcal H^2.
\]
\end{proposition}

\begin{proof}
By Lemma~\ref{lem:critical-weighted-transition}, from \eqref{eq:critical-open-set-liminf} with $U=N_{\delta,\rho}$,
\[
    \liminf_{\varepsilon\to0}\widetilde{\mathcal{T}}_{\varepsilon,\eta,\theta}(N_{\delta,\rho})\geq2\sqrt{w_p}\,|\Lambda_\eta|(N_{\delta,\rho})-C_{\eta,\theta}|N_{\delta,\rho}|.
\]

From \eqref{eq:critical-Dg-decomposition} and \eqref{eq:critical-ADg-decomposition},
\[
    |\Lambda_\eta|\mres\partial A=c_{\ell,\eta}(Q^-)G(Q^-,\nu_A)\,\mathcal H^2\mres\partial A.
\]
Since $\partial A\subset N_{\delta,\rho}$,
\[
    2\sqrt{w_p}\,|\Lambda_\eta|(N_{\delta,\rho})\geq\int_{\partial A}m_\ell^\eta(Q^-,\nu_A)\,d\mathcal H^2
\]
by \eqref{eq:critical-truncated-density}. This proves the first assertion.

Finally, $|N_{\delta,\rho}|\to0$ as $\delta,\rho\downarrow0$, while $m_\ell^\eta(Q^-,\nu_A)\uparrow m_\ell(Q^-,\nu_A)$ as $\eta\downarrow 0$. The second assertion follows by monotone convergence.
\end{proof}

\subsection{Proof of the lower-bound inequality}
We conclude by combining the localized estimates obtained in the
preceding sections.

\begin{proposition}[Lower bound]
\label{prop:lower-bound}
Let $r_\eps:=\kappa_\eps/\eps\to r\in[0,\infty]$, and let $(Q_\eps,\phi_\eps)\in\mathcal A_\eps$ be an energy-bounded sequence such that $(Q_\eps,\phi_\eps)\to(Q,A)$ with $(Q,A)\in\mathcal A_0^{\rm reg}$. Then
\[
    \liminf_{\eps\to0}E_{\eps,\kappa_\eps}(Q_\eps,\phi_\eps)\ge E_0^r(Q,A).
\]
\end{proposition}

\begin{proof}
We distinguish the three scaling regimes.

\medskip
\noindent\emph{Case $r=+\infty$.}
Fix $\eta\in(0,1/2)$ and
$0<\delta,\rho<\rho_0$. By nonnegativity of the omitted terms,
\[
    E_{\varepsilon,\kappa_\varepsilon}\geq E_\varepsilon^{\mathrm{bulk}}(A_\delta)+E_\varepsilon^{\mathrm{mix}}+R_\varepsilon^{\delta,\rho,\eta,1}.
\]
Using Propositions~\ref{prop:bulk-lower-bound} and \ref{prop:mixing-lower-bound}, together with \eqref{eq:localized-exterior-relaxation}, and then letting $\delta,\rho\downarrow0$ and $\eta\downarrow0$, we obtain
\begin{align*}
    &\liminf_{\varepsilon\to0}E_{\varepsilon,\kappa_\varepsilon}(Q_\varepsilon,\phi_\varepsilon)\geq{}\\
    &\int_A\left(|\nabla Q|^2+\frac{1}{2\lambda}F_b(Q)\right)\,dx+\int_{\partial A}\left(\frac{w_p}{3}G(Q^-,\nu_A)+\frac{w_v}{2}|Q^-|^2\right)\,d\mathcal H^2=E_0^\infty(Q,A).
\end{align*}

\medskip
\noindent\emph{Case $r=0$.}
By Lemma~\ref{lem:strong-anchoring}, $Q^-=0$ $\mathcal{H}^2-$ a.e. on $\partial A$. For every $0<\rho<\rho_0$,
\[
    E_{\varepsilon,\kappa_\varepsilon}\geq E_\varepsilon^{\mathrm{bulk}}(A_\rho)+E_\varepsilon^{\mathrm{mix}}.
\]
Hence, by
Propositions~\ref{prop:bulk-lower-bound} and \ref{prop:mixing-lower-bound}, followed by $\rho\downarrow0$,
\begin{align*}
    \liminf_{\varepsilon\to0}E_{\varepsilon,\kappa_\varepsilon}(Q_\varepsilon,\phi_\varepsilon)\geq{}&\int_A\left(|\nabla Q|^2+\frac{1}{2\lambda}F_b(Q)\right)\,dx+\frac{w_p}{3}\int_{\partial A}G(0,\nu_A)\,d\mathcal H^2\\
    ={}&E_0^0(Q,A).
\end{align*}

\medskip
\noindent\emph{Case $r=\ell\in(0,\infty)$.}
Fix $\eta\in(0,1/2)$ and $\theta\in(0,1)$, and set \(\beta:=1-\theta\). The restriction \(\eta<1/2\) allows us to apply Proposition~\ref{prop:localized-exterior-relaxation}.

The sets \(A_\delta\) and \(N_{\delta,\rho}\) are disjoint. On \(N_{\delta,\rho}\cap\{\phi_\varepsilon>\eta\}\), the coefficients \(\theta\) and \(\beta=1-\theta\) appearing in \(\widetilde{\mathcal T}_{\varepsilon,\eta,\theta}\) and \(R_\varepsilon^{\delta,\rho,\eta,\beta}\), respectively, sum to one and therefore use exactly the available bulk-gradient energy. On \(N_{\delta,\rho}\cap\{\phi_\varepsilon\leq\eta\}\), \(R_\varepsilon^{\delta,\rho,\eta,\beta}\) uses only the available void energy. All other diffuse-energy terms are nonnegative and may be discarded. Consequently,
\[
    E_{\varepsilon,\kappa_\varepsilon}\geq E_\varepsilon^{\mathrm{bulk}}(A_\delta)+\widetilde{\mathcal T}_{\varepsilon,\eta,\theta}(N_{\delta,\rho})+R_\varepsilon^{\delta,\rho,\eta,\beta}.
\]

Combining Proposition~\ref{prop:bulk-lower-bound}, \eqref{eq:critical-interface-fixed-eta}, and \eqref{eq:localized-exterior-relaxation}, 
\begin{align*}
    \liminf_{\varepsilon\to0}E_{\varepsilon,\kappa_\varepsilon}\geq&\int_{A_\delta}\left(|\nabla Q|^2+\frac{1}{2\lambda}F_b(Q)\right)\,dx+\int_{\partial A}m_\ell^\eta(Q^-,\nu_A)\,d\mathcal H^2\\
    &-C_{\eta,\theta}|N_{\delta,\rho}|+c_{\delta,\rho}\mu_{\delta,\eta,\beta}B_\delta(Q),
\end{align*}
and then letting $\delta,\rho\downarrow0$, gives
\begin{align*}
    \liminf_{\varepsilon\to0}E_{\varepsilon,\kappa_\varepsilon}(Q_\varepsilon,\phi_\varepsilon)\geq{}&\int_A\left(|\nabla Q|^2+\frac{1}{2\lambda}F_b(Q)\right)\,dx\\
    &+\int_{\partial A}\left(m_\ell^\eta(Q^-,\nu_A)+\frac{w_v}{2}(1-\eta)|Q^-|^2\right)\,d\mathcal H^2.
\end{align*}
Finally, $m_\ell^\eta(Q^-,\nu_A)\uparrow m_\ell(Q^-,\nu_A)$ as $\eta\downarrow 0$. Thus, by monotone convergence,
\begin{align*}
    &\liminf_{\varepsilon\to0}E_{\varepsilon,\kappa_\varepsilon}(Q_\varepsilon,\phi_\varepsilon)\\
    &\quad\geq\int_A\left(|\nabla Q|^2+\frac{1}{2\lambda}F_b(Q)\right)\,dx+\int_{\partial A}\left(m_\ell(Q^-,\nu_A)+\frac{w_v}{2}|Q^-|^2\right)\,d\mathcal H^2=E_0^\ell(Q,A).
\end{align*}
\end{proof}

\section{Recovery Sequences and the Upper Bound}\label{sec:upper-bound}
We construct recovery sequences in the three scaling regimes. The phase transition is placed in an $O(\varepsilon)$-neighborhood of a slightly shifted interface and is described by a boundary-dependent one-dimensional profile realizing the transition density $m_r$. For $r>0$, the tensor is then relaxed to zero on the exterior scale $\kappa_\varepsilon$, whereas for $r=0$ it is cut off before the phase-transition layer.
\subsection{Uniform one-dimensional profiles}
\label{sec:profiles}
For $p=(q,\nu)\in\mathcal S_0\times\mathbb S^2$, set
\begin{equation*}
    a(p):=w_pG(q,\nu)^2, \qquad B_\ell(p,s):=w_pW(s)+\frac{w_v}{2\ell}|q|^2(1-s)^2.
\end{equation*}

In the regime $r=+\infty$, define
\begin{equation}\label{eq:classical-profile-energy}
    J_\infty(\eta;p):=\int_{\mathbb R}\left(a(p)|\eta'|^2+w_pW(\eta)\right)\,dz,
\end{equation}
subject to $\eta(-\infty)=1$ and $\eta(+\infty)=0$. The arithmetic--geometric mean inequality gives
\begin{equation*}
    J_\infty(\eta;p)\geq2\int_0^1\sqrt{a(p)w_pW(s)}\,ds=m_\infty(p).
\end{equation*}
Equality is attained by the decreasing equipartition profile $\eta(z)=1/(1+\exp(z/G(p)))$.

For \(R>2\), define
\[
    \eta_R(p,z):=\frac{1}{1+\exp\!\left((z+R/2)/G(p)\right)}.
\]
Let $\vartheta(s):=3s^2-2s^3$ for $s\in [0,1]$, and define
\[
    \eta_{\infty,R}(p,z):=
\begin{cases}
    1,& z\le -R,\\[1mm]
    1-\bigl(1-\eta_R(p,-R+1)\bigr)\vartheta(z+R),& -R<z<-R+1,\\[1mm]
    \eta_R(p,z),& -R+1\le z\le -1,\\[1mm]
    \eta_R(p,-1)\bigl(1-\vartheta(z+1)\bigr),& -1<z<0,\\[1mm]
    0,& z\ge0.
\end{cases}
\]
Then $\eta_{\infty,R}\in H^1(\mathbb{R};[0,1])$, $\eta_{\infty,R}=1$ on $(-\infty,-R]$ and $\eta_{\infty,R}=0$ on $[0,\infty)$. 

\begin{lemma}\label{lem:uniform-classical-profile}
Let $K\subset \mathcal S_0\times \mathbb S^2$ be compact. For \(p=(q,\nu)\in K\), let \(\eta_{\infty,R}(p,\cdot)\) be the profile defined above. Then, for every \(R>2\), $\eta_{\infty,R}\in\operatorname{Lip}\bigl(K\times[-R,0];[0,1]\bigr)$, and, for every \(p\in K\),
\[
    \eta_{\infty,R}(p,-R)=1,\qquad\eta_{\infty,R}(p,0)=0,\qquad\partial_z\eta_{\infty,R}(p,z)\le0\quad\text{for a.e. }z\in(-R,0).
\]

Moreover, there exist constants \(C_K,c_K>0\), depending only on \(K\), such that
\[
    0\le J_\infty\bigl(\eta_{\infty,R}(p,\cdot);p\bigr)-m_\infty(p)\le C_K e^{-c_KR}
\]
for every \(p\in K\) and every \(R>2\). In particular,
\[
    \sup_{p\in K}\left|J_\infty\bigl(\eta_{\infty,R}(p,\cdot);p\bigr)-m_\infty(p)\right|\to0\qquad\text{as }R\to+\infty.
\]
\end{lemma}

\begin{proof}
Since \(K\) is compact and \(G\) is continuous, there exists \(G_K<+\infty\) such that \(1\le G(p)\le G_K\) for every \(p\in K\). The two discarded tails of the translated logistic profile are bounded uniformly by $C_Ke^{-c_KR}$. Since $W(s)\leq Cs^2$ near $0$ and $W(s)\leq C(1-s)^2$ near $1$, the interpolation and discarded-tail energies satisfy the same bound. Hence
\[
    0\leq J_\infty(\eta_{\infty,R}(p,\cdot);p)-m_\infty(p)\leq C_Ke^{-c_KR}.
\]
The asserted Lipschitz dependence follows from the Lipschitz continuity of $G$ on $K$ and the explicit definition of $\eta_{\infty,R}$.
\end{proof}

We next consider the critical regime $r=\ell\in(0,\infty)$. For $R>0$, set
\[
    \mathcal P_R:=\left\{\eta\in H^1((-R,0);[0,1]):\eta(-R)=1,\ \eta(0)=0\right\},
\]
and define
\begin{equation*}
    J_{\ell,R}(\eta;p):=\int_{-R}^{0}\left(a(p)|\eta'|^2+B_\ell(p,\eta)\right)\,dz.
\end{equation*}
By the arithmetic--geometric mean inequality and the chain rule,
\begin{align*}
    J_{\ell,R}(\eta;p)&\geq2\int_{-R}^0\sqrt{a(p)B_\ell(p,\eta)}\,|\eta'|\,dz\geq2\int_0^1\sqrt{a(p)B_\ell(p,s)}\,ds=m_\ell(p).
\end{align*}

When $q\neq0$, however, $B_\ell(p;0)=w_v|q|^2/(2\ell)>0$. Consequently, a whole-line profile converging to $0$ as $z\to+\infty$ would have infinite energy. We therefore use finite-interval near-minimizers that terminate exactly at $\eta=0$.

\begin{lemma}[Uniform critical profiles on compact parameter sets]
\label{lem:uniform-critical-profiles}
Let $K\subset \mathcal S_0\times \mathbb S^2$ be compact, and let $\delta>0$. Then there exist $R_\delta>0$ and a function $\widehat\eta_\delta\in\operatorname{Lip}\bigl(K\times[-R_\delta,0];[0,1]\bigr)$ such that, for every $p=(q,\nu)\in K$,
\[
    \widehat\eta_\delta(p,-R_\delta)=1,\qquad\widehat\eta_\delta(p,0)=0,\qquad\partial_z\widehat\eta_\delta(p,z)\le 0
\]
for a.e. $z\in(-R_\delta,0)$, and
\[
    \int_{-R_\delta}^{0}\left(a(p)|\partial_z\widehat\eta_\delta(p,z)|^2+B_\ell\bigl(p,\widehat\eta_\delta(p,z)\bigr)\right)\,dz\le m_\ell(p)+\delta.
\]
\end{lemma}
\begin{proof}
By compactness of \(K\), there exist constants $0<a_*\leq a(p)\leq a^*<+\infty$ for every $p\in K$. Fix \(\alpha,\beta\in(0,1/2)\). Since \(B_\ell\) is continuous and strictly positive on the compact set \(K\times[\alpha,1-\beta]\), there exist constants \(0<b_*\leq b^*<+\infty\) such that $b_*\leq B_\ell(p,s)\leq b^*$ for every $(p,s)\in K\times[\alpha,1-\beta]$. Moreover, $f(p,s):=\sqrt{a(p)/B_\ell(p,s)}$ is Lipschitz on \(K\times[\alpha,1-\beta]\), and there exist constants \(0<c_*\leq c^*<+\infty\) such that $c_*\leq f(p,s)\leq c^*$ on this set.

For \(p\in K\), define
\[
    \Xi_p(s):=\int_\alpha^s f(p,t)\,dt,\qquad L_p:=\Xi_p(1-\beta).
\]
Then \(p\mapsto L_p\) is Lipschitz on \(K\). Furthermore, \(\Xi_p\) is strictly increasing and satisfies $$c_*|s-\widetilde s|\leq|\Xi_p(s)-\Xi_p(\widetilde s)|\leq c^*|s-\widetilde s|.$$ Let $S_p:[0,L_p]\to[\alpha,1-\beta]$ denote the inverse of \(\Xi_p\). The family of inverse maps is jointly Lipschitz on $\mathcal D_{\alpha,\beta}:=\{(p,u):p\in K,\ 0\leq u\leq L_p\}$. Indeed, for \((p,u),(\widetilde p,\widetilde u) \in\mathcal D_{\alpha,\beta}\), the preceding lower bound and the Lipschitz dependence of \(\Xi_p\) on \(p\) give
\[
    |S_p(u)-S_{\widetilde p}(\widetilde u)|\leq C_{\alpha,\beta,K}\bigl(|p-\widetilde p|+|u-\widetilde u|\bigr).
\]

Let $L_{\max}:=\max_{p\in K}L_p$ and $R_{\alpha,\beta}:=1+\alpha+L_{\max}$, and let \(\vartheta(s):=3s^2-2s^3\) for \(s\in[0,1]\). For \(p\in K\), define
\[
\eta_{\alpha,\beta}(p,z):=
\begin{cases}
    1, &-R_{\alpha,\beta}\leq z\leq-L_p-\alpha-1, \\[0.4em]
    1-\beta\, \vartheta\bigl(z+L_p+\alpha+1\bigr), &-L_p-\alpha-1<z<-L_p-\alpha,\\[0.4em]
    S_p(-z-\alpha), &-L_p-\alpha\leq z\leq-\alpha,\\[0.4em]
    \alpha\left( 1-\vartheta\left(\dfrac{z+\alpha}{\alpha}\right)\right), &-\alpha<z<0.
\end{cases}
\]
The values of the neighboring pieces agree at all the junction points. In addition, every piece is nonincreasing in \(z\), and hence $\partial_z\eta_{\alpha,\beta}(p,z)\leq0$ for almost every \(z\in(-R_{\alpha,\beta},0)\).

The Lipschitz dependence of \(L_p\), the joint Lipschitz continuity of \(S_p\), and the uniform Lipschitz bounds on the cubic interpolation pieces imply $\eta_{\alpha,\beta}\in\operatorname{Lip}\bigl(K\times[-R_{\alpha,\beta},0];[0,1]\bigr)$. Here the fact that the definitions agree at the parameter-dependent junction points ensures that the Lipschitz estimate remains valid across the moving interfaces.

On the central interval, $a(p)|\partial_z\eta_{\alpha,\beta}|^2=B_\ell(p,\eta_{\alpha,\beta})$; hence its energy is
\[
    2\int_\alpha^{1-\beta}\sqrt{a(p)B_\ell(p,s)}\,ds.
\]

On the left interpolation interval, one has \(1-\beta\leq\eta_{\alpha,\beta}\leq1\). Compactness of \(K\) gives $B_\ell(p,\eta_{\alpha,\beta})\leq C_K\beta^2$ and $|\partial_z\eta_{\alpha,\beta}|\leq C\beta$. Thus the energy of the left interpolation interval is bounded by \(C_K\beta^2\).

On the right interpolation interval, whose length is \(\alpha\), one has $0\leq\eta_{\alpha,\beta}\leq\alpha$ and $|\partial_z\eta_{\alpha,\beta}|\leq C$. Since \(B_\ell\) is uniformly bounded on \(K\times[0,1]\), the energy of this interval is bounded by \(C_K\alpha\). Consequently,
\begin{align*}
    J_{\ell,R_{\alpha,\beta}}\bigl(\eta_{\alpha,\beta}(p,\cdot);p\bigr)&\leq 2\int_\alpha^{1-\beta}\sqrt{a(p)B_\ell(p,s)}\,ds+C_K(\alpha+\beta^2)\\
    &\leq m_\ell(p)+C_K(\alpha+\beta^2)
\end{align*}
uniformly for \(p\in K\).

Choose \(\alpha,\beta\in(0,1/2)\) sufficiently small that $C_K(\alpha+\beta^2)\leq\delta$. Setting $R_\delta:=R_{\alpha,\beta}$ and $\widehat\eta_\delta:=\eta_{\alpha,\beta}$ proves the result.
\end{proof}

When $r=0$, finite limiting energy requires $q=0$. Since $G(0,\nu)=\left(1+\omega s_+^2/9\right)^{1/2}$ is independent of $\nu$, Lemma~\ref{lem:uniform-classical-profile}, applied with $q=0$, provides a single truncated logistic profile that may be used at every boundary point.

\subsection{Boundary-dependent profiles}
\label{sec:boundary-profiles}
We now evaluate the uniform profile families along the boundary data.

\begin{corollary}[Boundary-dependent profile selection] \label{lem:boundary-profile-selection}
Let $A\Subset\Omega$ have $C^2$ boundary and let $Q^-\in C^1(\partial A;\mathcal S_0)$. For every $\delta>0$, there exist $R_\delta>0$ and $\eta_\delta\in W^{1,\infty}\bigl(\partial A\times(-R_\delta,0);[0,1]\bigr)$ such that
\[
    \eta_\delta(y,-R_\delta)=1,\qquad\eta_\delta(y,0)=0,\qquad\partial_z\eta_\delta\le0.
\]
Moreover, the following regime-dependent estimates hold.

If \(r=+\infty\), then
\[
    w_p\int_{\partial A}\int_{-R_\delta}^{0}\left(G(Q^-,\nu_A)^2|\partial_z\eta_\delta|^2+W(\eta_\delta)\right)\,dz\,d\mathcal H^2\le\int_{\partial A}m_\infty(Q^-,\nu_A)\,d\mathcal H^2+\delta.
\]

If \(r=\ell\in(0,+\infty)\), then
\begin{align*}
    &\int_{\partial A}\int_{-R_\delta}^{0}\biggl(w_pG(Q^-,\nu_A)^2|\partial_z\eta_\delta|^2+w_pW(\eta_\delta)+\frac{w_v}{2\ell}|Q^-|^2(1-\eta_\delta)^2\biggr)\,dz\,d\mathcal H^2\\
    &\qquad\le\int_{\partial A}m_\ell(Q^-,\nu_A)\,d\mathcal H^2+\delta.
\end{align*}

If \(r=0\) and \(Q^-=0\), then \(\eta_\delta\) may be chosen independent of \(y\), and
\[
    w_p\int_{\partial A}\int_{-R_\delta}^{0}\left[G(0,\nu_A)^2|\partial_z\eta_\delta|^2+W(\eta_\delta)\right]\,dz\,d\mathcal H^2\le\int_{\partial A}m_0(0,\nu_A)\,d\mathcal H^2+\delta.
\]
\end{corollary}
\begin{proof}
Set $p(y):=(Q^-(y),\nu_A(y))$ and $K:=p(\partial A)$. Since $Q^-\in C^1(\partial A;\mathcal S_0)$ and $\partial A$ is $C^2$, the map $p$ is $C^1$ and $K$ is compact. Set $\delta_0:=\delta/(1+\mathcal H^2(\partial A))$.

If $r=+\infty$, Lemma~\ref{lem:uniform-classical-profile} allows us to choose $R_\delta>2$ such that
\[
    \sup_{p\in K}\left(J_\infty(\eta_{\infty,R_\delta}(p,\cdot);p)-m_\infty(p)\right)\leq\delta_0.
\]
We then set $\eta_\delta(y,z):=\eta_{\infty,R_\delta}(p(y),z)$.

If $r=\ell\in(0,\infty)$, apply Lemma~\ref{lem:uniform-critical-profiles} with error $\delta_0$, and define $\eta_\delta(y,z):=\widehat\eta_{\delta_0}(p(y),z)$.

Finally, if $r=0$, apply Lemma~\ref{lem:uniform-classical-profile} to $K_0:=\{(0,\nu):\nu\in\mathbb S^2\}$. Since $G(0,\nu)$ is independent of $\nu$, the resulting profile may be chosen independent of $y$.

In each case, the asserted $W^{1,\infty}$ regularity follows from the corresponding uniform Lipschitz dependence and the chain rule. Integrating the pointwise profile estimate over $\partial A$ gives the claimed bound.
\end{proof}

\subsection{Construction of the recovery sequence}
\label{sec:recovery}
\begin{proposition}[Recovery sequence]
    Let $(Q,A)\in\mathcal A_0^{\rm reg}$ satisfy $E_0^r(Q,A)<\infty$. Then there exists $(Q_\eps,\phi_\eps)\in\mathcal A_\eps$ such that
\[
    \phi_\eps\to\chi_A \quad\text{strongly in }L^1(\Omega),\qquad Q_\eps\to Q \quad\text{strongly in }L^2(\Omega;\mathcal{S}_0),
\]
and
\[
    \limsup_{\eps\to0}E_{\eps,\kappa_\eps}(Q_\eps,\phi_\eps)\leq E_0^r(Q,A).
\]
\end{proposition}
\begin{proof}
    We first prove the statement for $Q\in C^1(\bar{A})$. The general case follows by approximation.

    Fix $\delta>0$, and let $R_\delta$ and $\eta_\delta$ be given by Corollary~\ref{lem:boundary-profile-selection}. Throughout the proof, $o_\varepsilon^\delta(1)$ denotes a quantity that converges to zero as $\varepsilon\to0$ for every fixed $\delta>0$. The constants implicit in $O_\delta(\cdot)$ and $o_\varepsilon^\delta(1)$ may depend on $\delta$, but are independent of $\varepsilon$. For a shift $\tau$, define in tubular coordinates
\[
    \phi_\epsilon^\tau(\Psi(y,t))=
\begin{cases}
    1, & t\leq \tau-\epsilon R_\delta,\\[0.2em]
    \eta_\delta\!\left(y,\dfrac{t-\tau}{\epsilon}\right),& \tau-\epsilon R_\delta<t<\tau,\\[0.5em]
    0, & t\geq\tau.
\end{cases}
\]
We extend $\phi_\epsilon^\tau$ by $1$ in the remaining interior region and by $0$ in the exterior. The map $M_\epsilon(\tau):=\int_\Omega\phi_\epsilon^\tau\,dx$ is continuous and increasing. Using tubular coordinates and the change of variables $t=\tau+\varepsilon z$, we obtain
\begin{align*}
    M_\varepsilon(0)-V_0=\varepsilon\int_{\partial A}\int_{-R_\delta}^0\bigl(\eta_\delta(y,z)-1\bigr)J_A(y,\varepsilon z)\,dz\,d\mathcal H^2(y),
\end{align*}
and hence $M_\varepsilon(0)-V_0=O_\delta(\varepsilon).$ Moreover, for almost every sufficiently small $\tau$,
\begin{align*}
    M_\varepsilon'(\tau)=\int_{\partial A}\int_{-R_\delta}^0-\partial_z\eta_\delta(y,z)J_A(y,\tau+\varepsilon z)\,dz\,d\mathcal H^2(y).
\end{align*}
Since $\int_{-R_\delta}^0-\partial_z\eta_\delta(y,z)\,dz=1$ uniformly for $y\in\partial A$, and $J_A(y,\tau+\varepsilon z)=1+O_\delta(\varepsilon+|\tau|)$, we obtain
\begin{equation}\label{eq:volume-correction-expansion}
    M_\varepsilon(0)-V_0=O_\delta(\varepsilon),\qquad M_\varepsilon^\prime(\tau)=\mathcal H^2(\partial A)+O_\delta(\varepsilon+|\tau|)
\end{equation}
uniformly for $|\tau|$ sufficiently small. Since \(0\leq\eta_\delta\leq1\), formula \eqref{eq:volume-correction-expansion} gives $M_\varepsilon(0)\leq V_0$ and $0\leq V_0-M_\varepsilon(0)\leq C_\delta\varepsilon$. Moreover, after decreasing the tubular neighborhood if necessary, for all sufficiently small \(\varepsilon\) and \(|\tau|\leq\tau_0\), $M_\varepsilon'(\tau)\geq\mathcal H^2(\partial A)/2>0$ for almost every \(\tau\).

Choose \(C_\delta'>0\) so large that $\mathcal H^2(\partial A)\,C_\delta'/2>C_\delta$. Then, for all sufficiently small \(\varepsilon\),
\begin{align*}
    M_\varepsilon(C_\delta'\varepsilon)-V_0=M_\varepsilon(0)-V_0+\int_0^{C_\delta'\varepsilon}M_\varepsilon'(\tau)\,d\tau\geq-C_\delta\varepsilon+\frac12\mathcal H^2(\partial A)C_\delta'\varepsilon>0.
\end{align*}
Since \(M_\varepsilon\) is continuous, the intermediate value theorem yields $\tau_\varepsilon\in[0,C_\delta'\varepsilon]$ such that $M_\varepsilon(\tau_\varepsilon)=V_0$. In particular, \(\tau_\varepsilon=O_\delta(\varepsilon)\).

We then set $\phi_\epsilon:=\phi_\epsilon^{\tau_\epsilon}$. Since $Q\in C^1(\overline A;\mathcal S_0)$ and $\partial A$ is $C^2$, we fix a $C^1$ extension of $Q$ to a tubular neighborhood of $\partial A$. We continue to denote this extension by $Q$.

Assume first that $r\in(0,\infty]$. Let $\zeta\in C^\infty(\mathbb R;[0,1])$ satisfy $\zeta=0$ on $(-\infty,0]$ and $\zeta=1$ on $[1,\infty)$. In tubular coordinates, define
\[
    Q_\eps(\Psi(y,t))=
    \begin{cases}
    Q(\Psi(y,t)), & t\leq\tau_\eps-\eps(R_\delta+1),\\[0.3em]
    (1-\zeta_\eps(t))Q(\Psi(y,t)) +\zeta_\eps(t)Q^-(y), & \tau_\eps-\eps(R_\delta+1)<t<\tau_\eps-\eps R_\delta,\\[0.3em]
    Q^-(y), & \tau_\eps-\eps R_\delta \leq t\leq\tau_\eps,\\[0.3em]
    Q^-(y)\iota_\eps(t-\tau_\eps),& \tau_\eps<t<\tau_\eps+2a_\eps,\\[0.3em]
    0, & t\geq\tau_\eps+2a_\eps,
\end{cases}
\]
where $\zeta_\eps(t):=\zeta\left((t-\tau_\eps+\eps(R_\delta+1))/\eps\right)$ and $a_\eps:=\sqrt{\kappa_\eps}$. Define, for $0\leq t\leq2a_\varepsilon$,
\[
    \iota_\varepsilon(t):=\frac{\sinh\bigl((2a_\varepsilon-t)/\kappa_\varepsilon\bigr)}{\sinh\bigl(2a_\varepsilon/\kappa_\varepsilon\bigr)}.
\]
Then $\iota_\varepsilon(0)=1$, $\iota_\varepsilon(2a_\varepsilon)=0$, and
\begin{equation}\label{eq:exterior-relaxation-profile}
    \int_0^{2a_\varepsilon}\left(\kappa_\varepsilon|\iota_\varepsilon^\prime|^2+\kappa_\varepsilon^{-1}|\iota_\varepsilon|^2\right)\,dt=\coth\left(\frac{2a_\varepsilon}{\kappa_\varepsilon}\right)=1+o(1),
\end{equation}
because $a_\varepsilon/\kappa_\varepsilon=\kappa_\varepsilon^{-1/2}\to\infty$.

If $r=0$, then finite limiting energy implies $Q^-=0$ $\mathcal H^2-$ a.e. on $\partial A$. Choose $h_\eps\downarrow0$ such that $\eps+|\tau_\eps|=o(h_\eps)$, and set
\[
    Q_\eps(x):=
    \begin{cases}
        \xi\left(-d_A(x)/h_\eps\right)Q(x), &x\in A,\\
        0, &x\in A^c,
    \end{cases}
\]
where $\xi\in C^{\infty}([0,\infty);[0,1])$ satisfies $\xi=0$ near $0$ and $\xi=1$ on $[1,\infty)$.

For $r>0$, we extend $Q_\varepsilon$ by $Q$ in the remaining interior region and by $0$ in the remaining exterior region.

In either case, the pieces match continuously across their common boundaries. Hence $(Q_\varepsilon,\phi_\varepsilon)\in H^1(\Omega;\mathcal S_0)\times H_0^1(\Omega)$. Moreover, $0\leq\phi_\varepsilon\leq1$ and $\int_\Omega\phi_\varepsilon\,dx=V_0$ by the choice of $\tau_\varepsilon$. Therefore, $(Q_\varepsilon,\phi_\varepsilon)\in\mathcal A_\varepsilon$.

For $r\in(0,\infty]$, introduce the three tubular regions
\begin{align*}
    I_\varepsilon^\delta&:=\Psi\Bigl(\bigl\{(y,t)\in\partial A\times\mathbb R:\tau_\varepsilon-\varepsilon R_\delta<t<\tau_\varepsilon\bigr\}\Bigr),\\
    J_\varepsilon^\delta&:=\Psi\Bigl(\bigl\{(y,t)\in\partial A\times\mathbb R:\tau_\varepsilon-\varepsilon(R_\delta+1)<t<\tau_\varepsilon-\varepsilon R_\delta\bigr\}\Bigr),\\
    O_\varepsilon&:=\Psi\Bigl(\bigl\{(y,t)\in\partial A\times\mathbb R:\tau_\varepsilon<t<\tau_\varepsilon+2a_\varepsilon\bigr\}\Bigr).
\end{align*}
Here $I_\varepsilon^\delta$ contains the phase-field transition, $J_\varepsilon^\delta$ is the inner matching region, and $O_\varepsilon$ is the exterior tensor-relaxation region.

All modifications of $\phi_\varepsilon$ are confined to $I_\varepsilon^\delta$, while the modifications of $Q_\varepsilon$ are confined to $J_\varepsilon^\delta\cup I_\varepsilon^\delta\cup O_\varepsilon$.

If $r\in(0,\infty]$, all modifications are supported in $J_\varepsilon^\delta\cup I_\varepsilon^\delta\cup O_\varepsilon$, whose volume tends to zero. Hence
\[
    \phi_\varepsilon\to\chi_A\quad\text{strongly in }L^1(\Omega),\qquad Q_\varepsilon\to Q\quad\text{strongly in }L^2(\Omega;\mathcal S_0).
\]

If $r=0$, the phase transition is again supported in a shrinking tubular neighborhood, while $Q_\varepsilon-Q$ is supported in an interior layer of width $O(h_\varepsilon)$. The same convergences therefore hold.

We next estimate the energy. Since $Q$ and its extension are $C^1$, $Q^-(y)-Q(\Psi(y,t))=O(|t|)$ uniformly near $\partial A$. On $J_\varepsilon^\delta$, $|t|=O_\delta(\varepsilon)$, and hence $|\zeta_\varepsilon'(t)|\,\left|Q^-(y)-Q(\Psi(y,t))\right|\leq C_\delta.$ Thus $\nabla Q_\varepsilon$ is uniformly bounded on $J_\varepsilon^\delta\cup I_\varepsilon^\delta$, and these sets have volume $O_\delta(\varepsilon)$. It follows that
\begin{equation}\label{eq:bulk-upper}
    E_\varepsilon^{\rm bulk}(Q_\varepsilon,\phi_\varepsilon)=\int_A\left(|\nabla Q|^2+\frac{1}{2\lambda}F_b(Q)\right)\,dx+o_\varepsilon^\delta(1).
\end{equation}

We now evaluate the energy on $I_\varepsilon^\delta$. There, $Q_\varepsilon(\Psi(y,t))=Q^-(y)$ and $\phi_\varepsilon(\Psi(y,t))=\eta_\delta\left(y,\frac{t-\tau_\varepsilon}{\varepsilon}\right).$ Using the change of variables $z=\frac{t-\tau_\varepsilon}{\varepsilon}$, the tubular-coordinate formula, and $J_A(y,\tau_\varepsilon+\varepsilon z)=1+o_\varepsilon^\delta(1)$ uniformly for $(y,z)\in\partial A\times[-R_\delta,0]$, we obtain
\begin{equation}\label{eq:mix-void-I}
\begin{aligned}
    &E_\varepsilon^{\rm mix}(Q_\varepsilon,\phi_\varepsilon;I_\varepsilon^\delta)+E_\varepsilon^{\rm void}(Q_\varepsilon,\phi_\varepsilon;I_\varepsilon^\delta)\\
    &=\int_{\partial A}\int_{-R_\delta}^0\Biggl[w_pG(Q^-,\nu_A)^2|\partial_z\eta_\delta|^2+w_pW(\eta_\delta)+\frac{w_v}{2r_\varepsilon}|Q^-|^2(1-\eta_\delta)^2\Biggr]\,dz\,d\mathcal H^2+o_\varepsilon^\delta(1).
\end{aligned}
\end{equation}
In tubular coordinates,
\[
    \nabla\phi_\varepsilon=\frac{1}{\varepsilon}\partial_z\eta_\delta\,\nu_A+\nabla_\tau\eta_\delta+O_\delta(\varepsilon),
\]
where $\nabla_\tau$ denotes the tangential gradient along $\partial A$. Since $\eta_\delta$ is uniformly Lipschitz for fixed $\delta$, after integration over the layer $I_\varepsilon^\delta$, the purely tangential terms are $O_\delta(\varepsilon^2)$ and the normal--tangential cross terms are $O_\delta(\varepsilon)$. Moreover, $Q_\varepsilon=Q^-(y)$ on $I_\varepsilon^\delta$, so that
\[
    \kappa_\varepsilon\int_{I_\varepsilon^\delta}|\nabla Q_\varepsilon|^2\,dx=O_\delta(\kappa_\varepsilon\varepsilon)=o_\varepsilon^\delta(1).
\]
Thus all the omitted terms in \eqref{eq:mix-void-I} are $o_\varepsilon^\delta(1)$.

If $r=\ell\in(0,\infty)$, then
$r_\varepsilon^{-1}\to\ell^{-1}$; if $r=\infty$, the last integrand in \eqref{eq:mix-void-I} converges uniformly to zero. Therefore, Corollary~\ref{lem:boundary-profile-selection} yields
\begin{equation}\label{eq:mix-void-I-2}
    \limsup_{\varepsilon\to0}\Bigl(E_\varepsilon^{\rm mix}(I_\varepsilon^\delta)+E_\varepsilon^{\rm void}(I_\varepsilon^\delta)\Bigr)\leq\int_{\partial A}m_r(Q^-,\nu_A)\,d\mathcal H^2+\delta.
\end{equation}

On $O_\varepsilon$, one has $\phi_\varepsilon=0$ and $Q_\varepsilon(\Psi(y,\tau_\varepsilon+s))=Q^-(y)\iota_\varepsilon(s)$ for $0<s<2a_\varepsilon$. Consequently,
\begin{equation}\label{eq:void-O}
    E_\varepsilon^{\rm void}(O_\varepsilon)=\frac{w_v}{2}\int_{\partial A}|Q^-|^2\int_0^{2a_\varepsilon}\left(\kappa_\varepsilon|\iota_\varepsilon'|^2+\kappa_\varepsilon^{-1}|\iota_\varepsilon|^2\right)J_A(y,\tau_\varepsilon+s)\,ds\,d\mathcal H^2+o_\varepsilon(1).
\end{equation}
The error in \eqref{eq:void-O} contains the tangential derivatives of $Q^-$. Since $Q^-\in C^1(\partial A;\mathcal S_0)$ and $a_\varepsilon=\sqrt{\kappa_\varepsilon}$,
\begin{align*}
    \kappa_\varepsilon\int_{O_\varepsilon}|\nabla_\tau Q_\varepsilon|^2\,dx\le C\kappa_\varepsilon\int_0^{2a_\varepsilon}|\iota_\varepsilon(s)|^2\,ds \le C\kappa_\varepsilon a_\varepsilon=C\kappa_\varepsilon^{3/2}\to 0.
\end{align*}
Furthermore,
\[
    \sup_{\substack{y\in\partial A\\0<s<2a_\varepsilon}} |J_A(y,\tau_\varepsilon+s)-1|\to0.
\]
Therefore, \eqref{eq:exterior-relaxation-profile} implies
\begin{equation}\label{eq:void-O-2}
    \lim_{\varepsilon\to0}E_\varepsilon^{\rm void}(O_\varepsilon)=\frac{w_v}{2}\int_{\partial A}|Q^-|^2\,d\mathcal H^2.
\end{equation}

Outside $I_\varepsilon^\delta$, the phase field is constant, so the mixing energy vanishes. The void potential is supported in $I_\varepsilon^\delta\cup O_\varepsilon$, while all remaining void-gradient contributions are $o_\varepsilon^\delta(1)$ because $\kappa_\varepsilon\to0$.

Combining \eqref{eq:bulk-upper}, \eqref{eq:mix-void-I-2}, and \eqref{eq:void-O-2}, we obtain, for $r\in(0,\infty]$,
\[
    \limsup_{\varepsilon\to0}E_{\varepsilon,\kappa_\varepsilon}(Q_\varepsilon,\phi_\varepsilon)\leq E_0^r(Q,A)+\delta.
\]

It remains to consider $r=0$. Since $Q\in C^1(\overline A;\mathcal S_0)$ and $Q^-=0$, continuity of the trace gives $Q=0$ on $\partial A$. Consequently,
\[
    |Q(\Psi(y,t))|\leq C|t|
\]
for $t<0$ sufficiently close to zero. Hence $\{\nabla Q_\varepsilon\}$ is uniformly bounded in the cutoff region, and
\begin{equation}\label{eq:bulk-upper-1}
    E_\varepsilon^{\rm bulk}(Q_\varepsilon,\phi_\varepsilon)=\int_A\left(|\nabla Q|^2+\frac{1}{2\lambda}F_b(Q)\right)\,dx+o_\varepsilon(1).
\end{equation}

Because $\varepsilon+|\tau_\varepsilon|=o(h_\varepsilon)$, the field $Q_\varepsilon$ vanishes on the support of $\nabla\phi_\varepsilon$ for all sufficiently small $\varepsilon$. Therefore,
\begin{align*}
    E_\varepsilon^{\rm mix}(Q_\varepsilon,\phi_\varepsilon)&=\int_{\partial A}\int_{-R_\delta}^0w_p\left[G(0,\nu_A)^2|\partial_z\eta_\delta|^2+W(\eta_\delta)\right]\,dz\,d\mathcal H^2+o_\varepsilon^\delta(1)\\
    &\leq\int_{\partial A}m_0(0,\nu_A)\,d\mathcal H^2+\delta+o_\varepsilon^\delta(1).
\end{align*}
Furthermore,
\begin{equation}\label{eq:void-upper}
    E_\varepsilon^{\rm void}(Q_\varepsilon,\phi_\varepsilon)=o_\varepsilon(1).
\end{equation}
Indeed, since $\varepsilon+|\tau_\varepsilon|=o(h_\varepsilon)$, the support of $Q_\varepsilon$ is separated from the phase-transition layer. Thus $\phi_\varepsilon=1$ wherever $Q_\varepsilon\neq0$, and the void potential vanishes. Moreover, the cutoff construction gives
\[
    \sup_{\varepsilon>0}\int_\Omega|\nabla Q_\varepsilon|^2\,dx<+\infty,
\]
so the void-gradient contribution is $o(1)$ because $\kappa_\varepsilon\to0$.

Combining \eqref{eq:bulk-upper-1}--\eqref{eq:void-upper} gives
\[
    \limsup_{\varepsilon\to0}E_{\varepsilon,\kappa_\varepsilon}(Q_\varepsilon,\phi_\varepsilon)\leq E_0^0(Q,A)+\delta.
\]

Choose $\delta_k\downarrow0$ and then $\varepsilon_k\downarrow0$ such that, for every $0<\varepsilon\le\varepsilon_k$, all the preceding estimates with $\delta=\delta_k$ have total error at most $2\delta_k$. Defining the recovery sequence with $\delta=\delta_k$ for $\varepsilon\in(\varepsilon_{k+1},\varepsilon_k]$ gives the required recovery sequence for smooth $Q$.

For a general $Q\in H^1(A;\mathcal S_0)$, assume first that $r>0$. Choose $Q_j\in C^1(\overline A;\mathcal S_0)$ such that $Q_j\to Q$ strongly in $H^1(A;\mathcal S_0)$. Since $H^1(A)\hookrightarrow L^4(A)$ in dimension three, $F_b(Q_j)\to F_b(Q)$ strongly in $L^1(A)$. Moreover, the trace theorem gives $Q_j^-\to Q^-$ strongly in $L^2(\partial A;\mathcal S_0)$. In particular, \(Q_j^-\to Q^-\) in measure on \(\partial A\). Since \(\gamma_r\) is continuous for \(r\in(0,\infty]\), we
have $\gamma_r(Q_j^-,\nu_A)\to\gamma_r(Q^-,\nu_A)$ in measure on \(\partial A\).

Furthermore, the quadratic-growth estimate $0\leq\gamma_r(q,\nu)\leq C_r(1+|q|^2)$ holds uniformly for \(\nu\in\mathbb S^2\). Strong \(L^2\)-convergence gives
\begin{align*}
    \bigl\||Q_j^-|^2-|Q^-|^2\bigr\|_{L^1(\partial A)}\leq\bigl(\|Q_j^-\|_{L^2(\partial A)}+\|Q^-\|_{L^2(\partial A)}\bigr) \|Q_j^--Q^-\|_{L^2(\partial A)}\to0.
\end{align*}
Consequently, the family \(\{\gamma_r(Q_j^-,\nu_A)\}_j\) is uniformly integrable on \(\partial A\). Vitali's theorem therefore yields strong $L^1$ convergence. Together with the bulk convergence, this proves
\[
    E_0^r(Q_j,A)\to E_0^r(Q,A).
\]

If $r=0$, finite energy gives $Q^-=0$, and hence $Q\in H_0^1(A;\mathcal S_0)$. In this case, choose $Q_j\in C_c^\infty(A;\mathcal S_0)$ such that $Q_j\to Q$ strongly in $H^1(A;\mathcal S_0)$.

Finally, choose $j(k)\to\infty$ so that
\[
    \|Q_{j(k)}-Q\|_{H^1(A)}+\bigl|E_0^r(Q_{j(k)},A)-E_0^r(Q,A)\bigr|\le\frac1k,
\]
and apply the preceding smooth recovery construction to $Q_{j(k)}$, with diffuse-interface error at most $1/k$. A final diagonal selection yields the asserted recovery sequence.
\end{proof}

\section{Conclusion}\label{sec:conclusion}
We have identified the simultaneous sharp-interface limit within a regular-interface class. The ratio $r_\eps=\kap/\eps$ selects the surface law: separated layers recover the iterated limit when $r_\eps\to\infty$, comparable layers produce a coupled density when $r_\eps\to \ell\in(0,\infty)$, and $r_\eps\to 0$ forces the interior trace of Q to vanish on the interface.

The proof combines truncated-$SBV^2$ compactness, identification of a $Q_\eps$-dependent transition measure, and a one-dimensional capacity comparison for exterior incursions. The present result assumes $Q \in H^1(A; \mathcal{S}_0)$ and $\partial A\in C^2$. Extending it to finite-perimeter droplets and $SBV^2$ tensor fields, including the effective density of possible internal jumps, remains open.

\appendix
\section{Stability of the one-dimensional capacity}
\begin{lemma}[Stability of the one-dimensional capacity]
\label{lem:capacity-stability}
Let \(H_n,H\subset(0,\delta)\) be measurable sets such that $\chi_{H_n}\to\chi_H$ strongly in $L^1(0,\delta)$. For fixed $\alpha_0,c,\kappa>0$ and $\eta\in(0,1/2)$, define
\[
    a_n:=\alpha_0\chi_{H_n}+c\kappa(1-\chi_{H_n}),\qquad b_n:=c\kappa^{-1}(1-\eta)^2(1-\chi_{H_n}),
\]
and
\[
    a:=\alpha_0\chi_H+c\kappa(1-\chi_H),\qquad b:=c\kappa^{-1}(1-\eta)^2(1-\chi_H).
\]
For \(q_0\in\mathbb R^m\), set
\[
    C_{H_n}(q_0):=\inf_{\substack{q\in H^1(0,\delta;\mathbb R^m)\\ q(0)=q_0}}\int_0^\delta\left(a_n|q'|^2+b_n|q|^2\right)\,dt,
\]
and define \(C_H(q_0)\) analogously. Then $C_{H_n}(q_0)\to C_H(q_0)$.
\end{lemma}
\begin{proof}
Since the functional is diagonal in the components of \(q\), it is enough to prove the result in the scalar case.

Let $X_{q_0}:=\{q\in H^1(0,\delta):q(0)=q_0\}$ and define
\[
    J_n(q):=\int_0^\delta\left(a_n|q'|^2+b_n|q|^2\right)\,dt,\qquad J(q):=\int_0^\delta\left(a|q'|^2+b|q|^2\right)\,dt.
\]
Then $C_{H_n}(q_0)=\min_{q\in X_{q_0}}J_n(q)$ and $C_H(q_0)=\min_{q\in X_{q_0}}J(q)$.

Let \(q_n\in X_{q_0}\) be a minimizer of \(J_n\). Using the constant function \(q\equiv q_0\) as a competitor gives $J_n(q_n)\le J_n(q_0)\le C|q_0|^2$. Since $a_n(t)\ge a_*:=\min\{\alpha_0,c\kappa\}>0$, we obtain
\[
    \int_0^\delta |q_n'|^2\,dt\le C.
\]
Together with \(q_n(0)=q_0\), this yields $\|q_n\|_{H^1(0,\delta)}\le C$. Hence, after passing to a subsequence, $q_n\rightharpoonup q$ weakly in $H^1(0,\delta)$ and $q_n\to q$ uniformly on $[0,\delta]$. In particular, \(q\in X_{q_0}\).

After passing to a further subsequence, we may assume that $a_n(t)\to a(t)$ and $b_n(t)\to b(t)$ for almost every \(t\in(0,\delta)\). Since \(a_n\) is uniformly bounded, for every \(\psi\in L^2(0,\delta)\), $\sqrt{a_n}\,\psi\to\sqrt a\,\psi$ strongly in $L^2(0,\delta)$. Therefore,
\[
    \sqrt{a_n}\,q_n'\rightharpoonup\sqrt a\,q'\qquad\text{weakly in }L^2(0,\delta),
\]
and weak lower semicontinuity gives
\[
    \int_0^\delta a|q'|^2\,dt\le\liminf_{n\to\infty}\int_0^\delta a_n|q_n'|^2\,dt.
\]

Moreover, the uniform convergence of \(q_n\), together with \(b_n\to b\) strongly in \(L^1(0,\delta)\), implies
\[
    \int_0^\delta b_n|q_n|^2\,dt\to\int_0^\delta b|q|^2\,dt.
\]
Consequently, $J(q)\le\liminf_{n\to\infty}J_n(q_n)$ and hence $C_H(q_0)\le\liminf_{n\to\infty}C_{H_n}(q_0)$.

For the reverse inequality, let $\bar q\in X_{q_0}$ be a minimizer of $J$. We first show that
\[
    \int_0^\delta a_n|\bar q'|^2\,dt\to\int_0^\delta a|\bar q'|^2\,dt.
\]
Set $a^*:=\max\{\alpha_0,c\kappa\}$. Then $|a_n-a|\le 2a^*$ a.e. in $(0,\delta)$. For every $K>0$, writing $f:=|\bar q'|^2\in L^1(0,\delta)$, we obtain
\begin{align*}
    \int_0^\delta |a_n-a|f\,dt\le K\|a_n-a\|_{L^1(0,\delta)}+2a^*\int_{\{f>K\}}f\,dt.
\end{align*}
For fixed $K$, the first term converges to zero as $n\to\infty$. The second term converges to zero as $K\to\infty$, by the absolute continuity of the integral of $f$. Therefore,
\[
    \int_0^\delta a_n|\bar q'|^2\,dt\to\int_0^\delta a|\bar q'|^2\,dt.
\]

Moreover, since $\bar q\in H^1(0,\delta)\hookrightarrow C^0([0,\delta])$, we have
\begin{align*}
    \left|\int_0^\delta(b_n-b)|\bar q|^2\,dt\right|\le\|\bar q\|_{L^\infty(0,\delta)}^2\|b_n-b\|_{L^1(0,\delta)}\to0.
\end{align*}
Consequently, $J_n(\bar q)\longrightarrow J(\bar q)$. Using $\bar q$ as a competitor for $J_n$, we obtain
\[
    \limsup_{n\to\infty} C_{H_n}(q_0)\le\lim_{n\to\infty}J_n(\bar q)=J(\bar q)=C_H(q_0).
\]
\end{proof}

\section*{Acknowledgments}
The authors thank Baoming Shi and Dawei Wu for helpful discussions.
\bibliographystyle{siamplain}
\bibliography{references}
\end{document}